\documentclass{article}

\usepackage{lineno,hyperref}
\modulolinenumbers[5]

\usepackage{amsmath,amssymb,amsthm, amscd, verbatim, setspace}
\usepackage{graphicx} 
\usepackage{url, hyperref}
\usepackage[section]{placeins}
\usepackage{tabu}
\usepackage{enumerate}
\usepackage{mathtools}
\usepackage[colorinlistoftodos,prependcaption,textsize=tiny]{todonotes}

\usepackage[utf8]{inputenc}
\usepackage[T1]{fontenc}

\makeatletter
\newtheorem*{rep@theorem}{\rep@title}
\newcommand{\newreptheorem}[2]{%
\newenvironment{rep#1}[1]{%
 \def\rep@title{#2~\ref{##1}}%
 \begin{rep@theorem}}%
 {\end{rep@theorem}}}
\makeatother

\theoremstyle{plain}
\newtheorem{theorem}{Theorem}[section]
\newtheorem{lemma}[theorem]{Lemma}

\newtheorem{thm}[theorem]{Theorem}
\newreptheorem{thm}{Theorem}

\newtheorem{lem}[theorem]{Lemma}
\newtheorem{prop}[theorem]{Proposition}
\newreptheorem{prop}{Proposition}
\newtheorem{cor}[theorem]{Corollary}
\newreptheorem{cor}{Corollary}
\newtheorem{conj}[theorem]{Conjecture}
\newreptheorem{conj}{Conjecture}

\theoremstyle{definition}
\newtheorem{definition}[theorem]{Definition}

\theoremstyle{definition}
\newtheorem{defn}[theorem]{Definition}

\theoremstyle{remark}

\newtheorem*{rem}{Remark}

\numberwithin{equation}{section}

\newcommand{\CC}{\mathbb{C}}
\newcommand{\DD}{\mathbb{D}}

\newcommand{\FF}{\mathbb{F}}

\newcommand{\N}{\mathbb{N}}

\newcommand{\RR}{\mathbb{R}}
\newcommand{\R}{\mathbb{R}}
\newcommand{\T}{\mathbb{T}}
\newcommand{\TT}{\mathbb{T}}
\newcommand{\U}{\mathbb{U}}
\newcommand{\UU}{\mathbb{U}}
\newcommand{\ZZ}{\mathbb{Z}}
\newcommand{\Z}{\mathbb{Z}}

\DeclareMathOperator{\ex}{ex}
\DeclareMathOperator{\BB}{BB}
\newcommand{\Esymb}{{\bf E}}
\DeclareMathOperator*{\E}{\Esymb}
\DeclareMathOperator{\rank}{rank}
\newcommand{\cB}{\mathcal B}

\newcommand{\expo}[1]{{\mathsf{e}\left(#1\right)}}

\usepackage{graphicx}
\usepackage{amssymb}
\usepackage{amsthm}

\begin{document}

\title{A Counting Lemma for Binary Matroids and Applications to Extremal Problems}

\author{Sammy Luo}
\date{November 2018}

\maketitle

\begin{abstract}
In graph theory, the Szemerédi regularity lemma gives a decomposition of the indicator function for any graph $G$ into a structured component, a uniform part, and a small error. This result, in conjunction with a counting lemma that guarantees many copies of a subgraph $H$ provided a copy of $H$ appears in the structured component, is used in many applications to extremal problems. An analogous decomposition theorem exists for functions over $\FF_p^n$. Specializing to $p=2$, we obtain a statement about the indicator functions of simple binary matroids. In this paper we extend previous results to prove a corresponding counting lemma for binary matroids. We then apply this counting lemma to give simple proofs of some known extremal results, analogous to the proofs of their graph-theoretic counterparts, and discuss how to use similar methods to attack a problem concerning the critical numbers of dense binary matroids avoiding a fixed submatroid.
\end{abstract}

\section{Introduction}
In this paper, the term \emph{matroid} refers to a simple binary matroid. A \emph{simple binary matroid} $M$ is, for our purposes, a subset of $\FF_2^r\setminus \{0\}$ having full rank in $\FF_2^r$. The positive integer $r=r(M)$ is called the \emph{rank} of $M$. The \emph{critical number} $\chi(M)$, first defined by Crapo and Rota~\cite{CrapoRota} under the name of \emph{critical exponent}, is the smallest $c$ such that there is a subspace of $\FF_2^r$ of codimension $c$ (i.e. a copy of $\FF_2^{r-c}$) contained in $\FF_2^r\setminus M$, or equivalently such that $M$ is contained in a union $A_1\cup\cdots\cup A_c$ where each $A_i$ is a hyperplane $\FF_2^r\setminus \FF_2^{r-1}$. For a classical reference for the basic theory of binary matroids, see \cite[Chapter~9]{Oxley}.

Basic examples of matroids include the following.

\begin{itemize}
    \item The \emph{projective geometry} of rank $r$, $PG(r-1,2)\coloneqq \FF_2^r\setminus \{0\}$, which has rank $r$ and critical number $r$.
    \item The \emph{affine geometry} of rank $r$, $AG(r-1,2)\coloneqq \FF_2^r\setminus \FF_2^{r-1}$, which has rank $r$ and critical number $1$.
    \item The \emph{Bose-Burton geometry}, $\BB(r,c)=\FF_2^r\setminus \FF_2^{r-c}$, a generalization of both examples above, which has rank $r$ and critical number $c$.
\end{itemize}

There is a direct connection between graphs and matroids: For any graph $G$ we can define its \emph{cycle matroid} $M(G)$, whose elements correspond to edges of $G$, where a set of elements of $M$ is linearly independent if and only if the corresponding edges of $G$ contain no cycle. A matroid is called \emph{graphic} if it is the cycle matroid of some graph. As observed in \cite{main}, if $\chi(G)$ is the chromatic number of $G$, we have $\chi(M(G))=\lceil \log_2(\chi(G))\rceil$.

The critical number of a matroid is analogous to the chromatic number of a graph. Just as the chromatic number plays a large role in many extremal problems in graph theory, the critical number plays a large role in extremal problems on matroids, which are often motivated by analogous problems for graphs. As there is a notion of a graph $G$ containing a copy of a subgraph $H$, there is a corresponding notion for matroids: a matroid $M$ \emph{contains} a copy of a matroid $N$ if there is a linear injection $\iota:\FF_2^{r(N)}\to \FF_2^{r(M)}$ such that $\iota(N)\subseteq M$. We often simply write this as $N\subseteq M$. Many extremal problems pose questions about criteria for the containment or avoidance of a fixed matroid $N$ in a matroid $M$.

One example of such an extremal problem is to determine the critical threshold of a matroid, another concept inspired by a graph-theoretical analogue.

\begin{defn}
Given a matroid $N$, the \emph{critical threshold} $\theta(N)$ is the infimum of all $\alpha>0$ for which there exists $c<\infty$ such that $|M|\geq \alpha 2^{r(M)}$ implies either $N\subseteq M$ or $\chi(M)\leq c$.
\end{defn}

The conjecture below is the extremal problem that motivates our work in this paper.
\begin{conj}[Geelen, Nelson, {\cite[Conj~1.7]{main}}]
\label{conj:mainsimp}
If $\chi(N)=c$, then $\theta(N)=1-i2^{-c}$, where $i\in \{2,3,4\}$.
\end{conj}

A more precise and technical version of this conjecture is stated as Conjecture~\ref{conj:main}. The technical details, and previous work towards solving the conjecture, are discussed in Section~\ref{sec:matrBack}.

The graph-theoretic analog of Conjecture~\ref{conj:mainsimp} is Theorem~\ref{thm:chromThresh}, which was proven in \cite{chromThresh}. The proof there makes use of the Szemerédi regularity lemma and a corresponding counting lemma. Roughly speaking, the regularity lemma states that, for any desired degree of uniformity $\varepsilon$, the vertices of a sufficiently large graph $G$ can be partitioned into a bounded number of parts of approximately the same size such that most (all but $\varepsilon$-fraction) pairs of parts $(X,Y)$ are $\varepsilon$-uniform, meaning that the edge density $d(X',Y')$ between large enough subsets $X',Y'$ of $X,Y$ does not differ too much from the edge density $d(X,Y)$ between $X$ and $Y$. Given such a partition $\Pi$ of $G$, we can construct a ``reduced graph'' $R=R_{\varepsilon,\delta}(\Pi)$ whose vertices are the parts in $\Pi$, with an edge between a pair $(X,Y)$ if and only if $(X,Y)$ is $\varepsilon$-uniform and $d(X,Y)\geq \delta$. The counting lemma states that for any graph $H$ contained as a subgraph in $R$, many copies of it are contained in $G$.

It is natural to consider approaching the critical threshold problem using analogous methods. Various regularity results analogous to the Szemerédi regularity lemma have been shown in the matroid setting, usually framed in terms of the indicator function for a matroid $M$ decomposing into several parts. The main such result we use, Theorem~\ref{thm:decomp}, is stated in Section~\ref{sec:decomps}. The statement of this theorem involves some technical terminology relating to nonclassical polynomial factors and Gowers norms, for which a brief introduction is given in Section~\ref{sec:decomps}.

Our main result in this paper is a corresponding Counting Lemma for matroids, Theorem~\ref{thm:counting}, which we develop by building on work in \cite{VeryCountingMaybe} and \cite{hatamiRegCount}. The case of the Counting Lemma for affine matroids (i.e. matroids with critical number $1$) is proved in \cite{VeryCountingMaybe}, and we adapt much of the same argument for our setting of more general matroids. Again, stating this Counting Lemma in a precise form requires building up technical definitions for concepts like the reduced matroid, based on the results of applying Theorem~\ref{thm:decomp}. The Counting Lemma and its proof can be found in Section~\ref{sec:count}.

In Section~\ref{sec:apply} we demonstrate a few simple applications of this Counting Lemma, giving short new proofs for the matroid analogues of the Removal Lemma and the Erd\H{o}s-Stone Theorem in graph theory. Finally, we discuss an approach to applying our Counting Lemma and related techniques to Conjecture~\ref{conj:mainsimp} in Section~\ref{sec:work}
. Along the way, we prove the following technical result, which has the Bose-Burton theorem stated in Section~\ref{sec:matrBack} as an immediate corollary and may be useful in other settings as well.

\begin{prop}
\label{prop:extBB}
Let $n,c$ be positive integers, let $k_1,\dots,k_n$ be nonnegative integers, and let $G = \bigoplus_{i=1}^n \frac{1}{2^{k_i+1}}\ZZ/\ZZ$. Let $H$ be a subgroup of $G$. Let $M_1,\dots,$ $M_{2^c-1}$ be subsets of $G$. Then there exist $H_1,\dots,H_c\in G/H$, cosets of $H$, such that for $1\leq i\leq c$,
\begin{equation}
    \frac{1}{|H|}\sum_{x\in\{0,1\}^{i-1}} \left | M_{2^{i-1}+\sum_{j=1}^{i-1} x_j 2^{j-1}}\cap \left( H_i + \sum_{j=1}^{i-1} x_j H_j \right ) \right | \geq \sum_{j=2^{i-1}}^{2^i-1}\frac{|M_j|}{|G|}.\tag{$*$}
\end{equation}
\end{prop}

\section{Extremal Problems on Graphs and Matroids}

We start by looking at a basic extremal problem on graphs, that of avoiding a fixed subgraph $H$.

\begin{defn}
The \emph{extremal number} for a graph $H$ and integer $n$ is defined by
$$\ex(H,n)=\max\{|E(G)| \mid |G|=n, H\not\subseteq G\}.$$
\end{defn}

The following theorem is a classical result.

\begin{thm}[Erd\H{o}s-Stone]
\label{thm:ES}
\[\lim_{n\rightarrow \infty} \frac{\text{ex}(H;n)}{|E(K_n)|}=1-\frac{1}{\chi(H)-1}.\]
\end{thm}

The special case where $H=K_m$ is a form of \emph{Turán's theorem}.

We can analyze the situation more carefully by looking for density thresholds above which, though graphs $G$ avoiding $H$ may exist, they are constrained by properties like a bounded chromatic number. It turns out that for graphs, the appropriate notion of density to consider here is the minimum degree $\delta(G)$ of a graph $G$.

\begin{defn}
\label{def:chromthresh}
Given a graph $H$, the \emph{chromatic threshold} $\theta(H)$ is the infimum of all $\alpha>0$ for which there exists $c<\infty$ such that $\delta(G)\geq \alpha |G|$ implies either $H\subseteq G$ or $\chi(G)\leq c$.
\end{defn}

The definition of the critical threshold of a matroid was motivated in analogy to Definition~\ref{def:chromthresh}.

The chromatic threshold was first determined for complete graphs in \cite{KnFree}, with an explicit sharp bound on the chromatic number involved.

\begin{thm}[Goddard, Lyle, {\cite[Thm~11]{KnFree}}]
If $\delta(G)>(2r-5)n/(2r-3)$ and $K_r\not\subseteq G$, then $\chi(G)\leq r+1$. In particular, $\theta(G)\leq \frac{2r-5}{2r-3}$.
\end{thm}

The chromatic threshold of a general graph $H$ was determined in the general case by Allen et al. in \cite{chromThresh}. To state the result, we first need to make the following definitions.

\begin{defn}
The \emph{decomposition family} $\mathcal{M}(H)$ of an $r$-partite graph $H$ is the set of bipartite graphs obtained by deleting all but $2$ color classes in some $r$-coloring of $H$.
\end{defn}

\begin{defn}
A graph $H$ is \emph{r-near-acyclic} if $\chi(H)=r$ and deleting all but $3$ color classes in some $r$-coloring of $H$ yields a graph $H'$ that can be partitioned into a forest $F$ and an independent set $S$ such that every odd cycle in $H'$ meets $S$ in at least $2$ vertices.
\end{defn}

Now we can state the main result of \cite{chromThresh}.

\begin{thm}[Allen, et al., {\cite[Thm~2]{chromThresh}}]
\label{thm:chromThresh}
If $\chi(H)=r$, then $\theta(H)=1-\frac{1}{r-\frac{i}{2}}$, where $i=2$ if and only if $\mathcal{M}(H)$ contains no forest, $i=4$ if and only if $H$ is $r$-near-acyclic, and $i=3$ otherwise.
\end{thm}

\label{sec:matrBack}

We can ask the same extremal questions for matroids.

\begin{defn}
The \emph{extremal number} for a matroid $N$ and integer $n$ is defined by
$$\ex(N,n)=\max\{|M| \mid r(M)=n, N\not\subseteq M\}.$$
\end{defn}

Note that if $\chi(N)=c$, then $N$ is contained in $BB(n,c)$ for some $n$. Geelen and Nelson prove the following analogue of the Erd\H{o}s-Stone theorem in \cite{gES}.

\begin{thm}[Geometric Erd\H{o}s-Stone, {\cite[Thm~1.2, $q=2$ case]{gES}}]
\label{thm:gES}
\[\lim_{n\rightarrow \infty} \frac{\ex(N,n)}{2^n-1}=1-2^{1-\chi(N)}.\]
\end{thm}

The $\chi(N)=1$ case is known as the \emph{Binary Density Hales-Jewett} theorem. In this case, Bonin and Qin in fact show that $\ex(AG(k,2),n)<2^{\alpha_k n + 1}$, where $\alpha_k=1-2^{-(k-1)}$ \cite[Lemma~21]{boninqin}.

The special case where $N=PG(c-1,2)$ is a form of the \emph{Bose-Burton theorem}, which has a more precise statement as follows.

\begin{thm}
\label{thm:BB}
If $M$ does not contain a copy of $PG(c-1,2)$, then $|M|\leq 2^{r(M)}-2^{r(M)-c+1}$.
\end{thm}

Note that taking $G=\FF_2^{r(M)}$, $H$ the trivial subgroup, and $M_1=\cdots=M_{2^c-1}=M$ in Proposition~\ref{prop:extBB} immediately yields Theorem~\ref{thm:BB}.

In \cite{tidor}, Tidor proved a result on the chromatic thresholds of projective geometries, analogous to Goddard and Lyle's result for complete graphs.

\begin{thm}[Tidor, {\cite[Thm~1.4]{tidor}}]
\label{thm:tidor}
If $|M|>(1-3\cdot 2^{-t})2^{r(M)}$ and $PG(t-1,2)\not\subseteq M$, then $\chi(M)\in \{t-1,t\}$. In particular, $\theta(PG(t-1,2))\leq 1-3\cdot 2^{-t}$.
\end{thm}

To formulate the precise version of Conjecture~\ref{conj:mainsimp}, we make the following definition.

\begin{defn}
A matroid $M$ is \emph{c-near-independent} if $\chi(M)=c$ and for some $(c-2)$-codimensional subspace $H$ with $\chi(M\cap H)=2$, $H$ has a 1-codimensional subspace $S$ such that $M\cap S$ is linearly independent, and every odd circuit in $M\cap H$ contains at least four elements of $H\setminus S$.
\end{defn}

Now we state the precise form of the conjecture.

\begin{conj}[Geelen, Nelson, {\cite[Conj~5.2]{main}}]
\label{conj:main}
If $\chi(N)=c$, then $\theta(N)=1-i2^{-c}$, where $i=2$ if and only if no $(c-1)$-codimensional subspace $S$ exists such that $S\cap N$ is a set of linearly independent vectors, $i=4$ if and only if $N$ is $c$-near-independent, and $i=3$ otherwise.
\end{conj}

In \cite{main}, Geelen and Nelson show that the conjectured expression is a valid lower bound.

\begin{thm}[Geelen, Nelson, {\cite[Thm~5.4]{main}}]
\label{thm:lower}
If $\chi(N)=c$, then $\theta(N)\geq 1-i2^{-c}$, where $i=2$ if and only if no $(c-1)$-codimensional subspace $S$ exists such that $S\cap N$ is a set of linearly independent vectors, and $i=4$ if and only if $N$ is $c$-near-independent, and $i=3$ otherwise.
\end{thm}

Combined with the trivial upper bound of $\theta(N)\leq 1 - 2^{1-c}$ that follows immediately from Theorem~\ref{thm:gES}, it remains to show that $\theta(N)\leq 1 - 3\cdot 2^{-c}$ when $N$ has a $(c-1)$-codimensional flat that is independent, and that $\theta(N)\leq 1 - 4\cdot 2^{-c}$ when $N$ is $c$-near-independent.

For $\ell\geq c+k-1$, $c>1$, define $N_{\ell,c,k}$ to be the rank $\ell$ matroid consisting of the union of $BB(\ell,c-1)$ with $k$ linearly independent vectors contained inside the complement of $BB(\ell,c-1)$ in $\FF_2^\ell$. This represents the most general maximal case of matroids of critical number $c$ for which $i$ is conjectured to be $3$. In Section~\ref{sec:work}, we verify Conjecture~\ref{conj:main} for $N_{\ell,2,1}$, the union of an affine geometry and a single other point. We then discuss how the tools in this paper could be applied in an approach to the general $N_{\ell,c,1}$ case.

\section{Regularity and Counting}

A \emph{regularity} or \emph{decomposition} result, in general, splits a generic object (e.g. a graph, a subset of an abelian group, or a function) into a structured part, a uniform part, and possibly a small error. A corresponding \emph{counting lemma} then guarantees that the number of copies of a suitable subobject contained in this object can be well-approximated by the number of copies contained in the structured part. The most well-known example of such a pair of results is the Szemerédi regularity lemma and the corresponding counting lemma for subgraphs contained in the reduced graph. As mentioned in the introduction, the use of this pair of lemmas is key to the argument used in \cite{chromThresh} to prove Theorem~\ref{thm:chromThresh}.

A simple example of an analogous regularity result for matroids is Green's regularity lemma (specialized to $\FF_2^n$). To state it, we make a few preliminary definitions.

\begin{defn}
Let $V=\FF_2^n$. A set $X\subset V$ is \emph{linearly $\varepsilon$-uniform} in $V$ if $|\widehat{1_X}(\xi)|\leq \varepsilon$ for all nonzero $\xi\in \hat{V}=V$, or equivalently if for each hyperplane $H\leq V$,

\[||X\cap H|-|X\setminus H||\leq \varepsilon |V|.\]
\end{defn}

\begin{defn}
Let $X\subseteq V=\FF_2^n$. A subspace $W\leq V$ is \emph{linearly $\varepsilon$-regular} with respect to $X$ if for all but $\varepsilon |V|$ values of $v\in V$, $(X-v)\cap W$ is linearly $\varepsilon$-uniform in $W$.
\end{defn}

Green's regularity result is the following.

\begin{thm}[Geometric Regularity Lemma, {\cite[Thm~2.1]{gReg}}]
\label{thm:greenReg}
For any $\varepsilon\in(0,\frac{1}{2})$ there is a $T>0$ such that for any $V=\FF_2^n$ and any subset $X\subset V$ there is a subspace $W\subseteq V$ of codimension at most $T$ that is linearly $\varepsilon$-regular with respect to $X$.

\end{thm}

This notion of regularity readily yields a counting lemma for triangles (and indeed, all odd circuits) in matroids, which Geelen and Nelson use in their proofs that $\theta(PG(1,2))\leq \frac{1}{4}$ \cite{main} and that odd circuits have critical threshold $0$ \cite{geelenOdd}. In Section~\ref{sec:vernm1}, we use a method along the same lines as their proof to verify Conjecture~\ref{conj:main} for $N_{\ell,2,1}$.

Unfortunately, the linear Fourier-analytic notion of regularity provided by Theorem~\ref{thm:greenReg} is not strong enough for a counting lemma to hold for general submatroids $N$. In attempting to translate the ideas of \cite{chromThresh} into tools for the matroid threshold problem, we therefore need a stronger regularity statement, one that admits a corresponding, more general counting lemma.

\subsection{Regularity on Matroids} \label{sec:decomps}
After the inverse conjecture for the Gowers norm over finite fields of low characteristic was established \cite{tao2012inverse}, stronger regularity results, using regularity with respect to the Gowers norms, came within reach. The primary regularity result that we will use is stated in \cite{VeryCountingMaybe} as a decomposition theorem for bounded functions on $\FF_p^n$. To work with this result, we will first need to introduce a few technical concepts from higher-order Fourier analysis.

Throughout this section, $p$ is taken to be a fixed prime. For our work with binary matroids, we will always take $p=2$. Given a function $f:\FF_2^n\rightarrow\{0,1\}$, in our case usually the indicator function of some matroid $M\subseteq \FF_2^n\setminus\{0\}$, the decomposition theorem will split it into a sum of three parts: a structured part, a uniform part, and a small error. Here we will address the technical issues that arise in working with the first two parts. In the sections below, we largely quote the terminology and notation used in \cite{VeryCountingMaybe} and \cite{hatamiRegCount}.

\subsubsection{The Gowers norm and nonclassical polynomials}

\begin{defn}
Given a function $f:\FF_p^n\rightarrow\CC$ and an integer $d\geq 1$, the \emph{Gowers norm of order $d$} for $f$ is
\[\|f\|_{U^d}=\left|\E_{h_1,\dots,h_d,x\in\FF_p^n}\left[\prod_{i_1,\dots,i_d\in\{0,1\}}\mathcal{C}^{i_1+\cdots+i_d}f\left(x + \sum_{j=1}^d i_j h_j\right)\right]\right|^{1/2^d},\]
where $\mathcal{C}$ denotes the conjugation operator.
\end{defn}

It is easy to see that $\|f\|_{U^d}$ is increasing in $d$ and is indeed a norm for $d\geq 2$, and that $\|f\|_{U^1}=|\E[f]|$ and $\|f\|_{U^2}=\|\hat{f}\|_{l^4}$. So, the Gowers norm of order $2$ is related to the Fourier bias used in Green's regularity lemma, a measure of correlation with exponentials of linear polynomials: $\|f\|_{U^2}$ is large if and only if $\sup_{\xi\neq 0}|\hat{f}(\xi)|$ is large, i.e. if and only if $f$ is strongly correlated with the exponential of some linear polynomial. It is natural to expect the Gowers norm of order $d+1$ to be similarly related to polynomials of degree $d$; conjectures that a large Gowers-$(d+1)$ norm implies correlation with the exponential of a degree $d$ polynomial, in various settings, were known as \emph{inverse conjectures} for the Gowers norms.

For large $d$ over fields of small characteristic, it turns out that the inverse conjectures are not true as stated; the right notion to consider, over which an inverse theorem for the Gowers norms actually holds, is that of a \emph{nonclassical polynomial}.

\begin{defn}
\label{defn:polynomial}
Let $\TT=\RR/\ZZ$. Given an integer $d\geq 0$, a function $P:\FF_p^n\rightarrow\TT$ is called a \emph{(non-classical) polynomial of degree at most $d$} if for all $h_1,\dots,h_d,x\in\FF_p^n$,
\[\sum_{i_1,\dots,i_d\in\{0,1\}}(-1)^{i_1+\cdots+i_d}P\left(x + \sum_{j=1}^d i_j h_j\right)=0.\]
\end{defn}

Since we will be working mostly with non-classical polynomials, it should be assumed that any use of the word ``polynomial'' refers to a possibly non-classical polynomial unless otherwise specified. 

Let $\expo{x}=e^{2\pi i x}$. It follows from definition that $\|f\|_{U^{d+1}}=1$ if and only if $f=\expo{P}$ for some non-classical polynomial $P$ of degree at most $d$, and $\|f\cdot \expo{P}\|_{U^{d+1}}=\|f\|_{U^{d+1}}$ for any function $f$ and non-classical polynomial $P$ of degree at most $d$. Non-classical polynomials can be characterized in terms of classical ones by the following lemma of Tao and Ziegler \cite{tao2012inverse}.

\begin{lem}[{\cite[Lemma~1.7]{tao2012inverse}}]\label{lem:nonclassrep}
Let $|\cdot|$ denote the standard map from $\FF_p$ to $\{0,1,\dots,p-1\}$. A function $P: \FF_p^n \to \T$ is a polynomial of degree at most $d$ if and
only if $P$ can be represented as
$$P(x_1,\dots,x_n) = \alpha + \sum_{\substack{0\leq d_1,\dots,d_n< p; k \geq 0:
  \\ 0 < \sum_i d_i \leq d - k(p-1)}} \frac{ c_{d_1,\dots, d_n,
  k} |x_1|^{d_1}\cdots |x_n|^{d_n}}{p^{k+1}} \mod 1,
$$
for a unique choice of $c_{d_1,\dots,d_n,k} \in \{0,1,\dots,p-1\}$
and $\alpha \in \T$.  The element $\alpha$ is called the {\em
  shift} of $P$, and the largest integer $k$ such that there
exist $d_1,\dots,d_n$ for which $c_{d_1,\dots,d_n,k} \neq 0$ is called
the {\em depth} of $P$. A depth-$k$ polynomial $P$ takes values in a coset of the subgroup $\U_{k+1}\coloneqq \frac{1}{p^{k+1}} \Z/\Z$. Classical polynomials correspond to
polynomials with $0$ shift and $0$ depth.
\end{lem}

For convenience we will assume henceforth that all polynomials have shift $0$, so that all polynomials of depth $k$ take values in  $\UU_{k+1}$; this will not affect our arguments.

\subsubsection{Polynomial factors and rank}

\begin{defn}
A \emph{polynomial factor} $\mathcal{B}$ of $\FF_p^n$ is a partition of $\FF_p^n$ into finitely many pieces, called \emph{atoms}, such that for some polynomials $P_1,\dots,P_C$, each atom is defined as the solution set $\{x|\forall i\in\{1,\dots,C\} P_i(x)=b_i\}$ for some $(b_1,\dots,b_C)\in \TT^C$. The \emph{complexity} of $\mathcal{B}$ is the number of defining polynomials $|\mathcal{B}|=C$, and the \emph{degree} is the highest degree among $P_1,\dots,P_C$. If $P_i$ has depth $k_i$, the \emph{order} of $\mathcal{B}$ is $\|\mathcal{B}\|=\prod_{i=1}^C p^{k_i+1}$, an upper bound on the number of atoms in $\mathcal{B}$.
\end{defn}

\begin{defn}
The \emph{$d$-rank} $\rank_d(P)$ of a polynomial $P$ is the smallest integer $r$ such that $P$ can be expressed as a function of $r$ polynomials of degree at most $d-1$. The \emph{rank} of a polynomial factor defined by $P_1,\dots,P_C$ is the least integer $r$ for which there is a tuple $(\lambda_1,\dots,\lambda_C)\in\ZZ^C$, with $(\lambda_1 \mod p^{k_1+1}, \ldots, \lambda_C \mod p^{k_C+1})\neq 0^C$, such that $\rank_d(\sum_{i=1}^C \lambda_iP_i) \leq r$, where $d=\max_i \deg(\lambda_i P_i)$.

Given a polynomial factor $\cB$ and a function $r:\Z_{>0}\rightarrow \Z_{>0}$ (possibly a constant, as in Proposition~\ref{prop:regUni} below), we say that $\cB$ is \emph{$r$-regular} if $\cB$ is of rank larger than $r(|\cB|)$.
\end{defn}

As described in~\cite{hatamiRegCount}, a polynomial of sufficiently high rank is, intuitively, one that is ``generic,'' with no unexpected decompositions into polynomials of lower degree. A sufficiently regular polynomial factor is then one whose constituent polynomials do not have any unexpected dependencies. This algebraic notion of being generic turns out to be related to an analytic notion of uniformity for polynomial factors.

\begin{defn}
For $\varepsilon>0$, we say that $\cB$ is \emph{$\varepsilon$-uniform} if for all $(\lambda_1,\dots,\lambda_C)\in\ZZ^C$ with $(\lambda_1 \mod p^{k_1+1}, \ldots, \lambda_C \mod p^{k_C+1})\neq 0^C$,
\[\left\|\expo{\sum_i \lambda_i P_i}\right\|_{U^d}<\varepsilon.\]
\end{defn}

Thus, a sufficiently uniform polynomial factor is one where linear combinations of the constituent polynomials are not strongly correlated with any lower-degree polynomials. The following fact, noted in \cite{hatamiRegCount}, follows from Theorem~1.20 of \cite{tao2012inverse}.

\begin{prop}\label{prop:regUni}
For every $\varepsilon>0,d\in\Z_{>0}$ there exists an integer $r=r(d,\varepsilon)$ such that every $r$-regular degree $d$ polynomial factor $\cB$ is $\varepsilon$-uniform.
\end{prop}

For our purposes, both regularity and uniformity can be treated as black boxes ensuring a polynomial factor is sufficiently pseudorandom to apply counting results in it; see Section~\ref{sec:count} for some results where these properties are relevant.

\subsubsection{The Strong Decomposition Theorem}
We can finally state the main regularity result we will use, the strong decomposition theorem of \cite{VeryCountingMaybe}.

\begin{thm}[Strong Decomposition Theorem, {\cite[Theorem~5.1]{VeryCountingMaybe}}]\label{thm:decomp}
Suppose $\delta > 0$ and $ d \geq 1$ are integers. Let $\eta: \N
\to \R^+$ be an arbitrary non-increasing function and $r: \N \to \N$ be an arbitrary
non-decreasing function. Then there exist $N =
N(\delta, \eta, r, d)$ and $C =
C(\delta,\eta,r,d)$ such that the following holds.

Given $f: \FF_p^n \to \{0,1\}$ where $n > N$, there
exist three functions $f_1, f_2, f_3: \FF_p^n \to
\R$ and a polynomial factor  $\cB$ of
degree at most $d$ and complexity at most $C$ such that the following conditions hold:
\begin{itemize}
\item[(i)]
$f=f_1+f_2+f_3$.
\item[(ii)]
$f_1 = \E[f|\cB]$, the expected value of $f$ on an atom of $\cB$.
\item[(iii)]
$\|f_2\|_{U^{d+1}} \leq \eta(|\cB|)$.
\item[(iv)]
$\|f_3\|_2 \leq \delta$.
\item[(v)]
$f_1$ and $f_1 + f_3$ have range $[0,1]$; $f_2$ and $f_3$ have range $[-1,1]$.
\item[(vi)]
$\cB$ is $r$-regular.
\end{itemize}
\end{thm}

In analogy with the terminology for the Szemerédi regularity lemma, we will call a decomposition of $f=1_M$ with the properties given by Theorem~\ref{thm:decomp} for parameters $\delta, \eta, r, d$ a \emph{$(\delta,\eta,r,d)$-regular partition of $f$ (or of $M$)}, and we say that $\cB$ is its \emph{corresponding factor}. Similarly, an \emph{$(\eta,r,d)$-regular partition of $f$ (or of $M$)} is the same thing with an unspecified value for $\delta$.

Theorem~\ref{thm:decomp} is strong enough to help us prove a corresponding counting lemma for general binary matroids.

\subsection{The Counting Lemma} \label{sec:count}
From this point onwards we will only be concerned with the field $\FF_2$, though most of the concepts also extend to prime-ordered fields $\FF_p$ in general. For convenience of notation, throughout this section let $[a,b]$ denote the set of integers from $a$ to $b$ inclusive, and let $[n]=[1,n]=\{1,\dots,n\}$.

Before stating our main result, the Counting Lemma, we first define the notion of a \emph{reduced matroid}, in analogy to the reduced graph used in the graph counting lemma.

\begin{defn}[Reduced Matroid]
Given a matroid $M\subseteq \FF_2^n\setminus \{0\}$ and an $(\eta,r,d)$-regular partition $f_1+f_2+f_3$ of $M$ with corresponding factor $\cB$, for any $\varepsilon,\zeta>0$ define the $(\varepsilon,\zeta)$-\emph{reduced matroid} $R=R_{\varepsilon, \zeta}$ to be the subset of $\FF_2^n$ whose indicator function $F$ is constant on each atom $b$ of $\cB$ and equals $1$ if and only if
\begin{enumerate}
    \item $\E[|f_3(x)|^2\mid x\in b]\leq \varepsilon^2$, and
    \item $\E[f(x)\mid x\in b]\geq \zeta$.
\end{enumerate}
\end{defn}

So, $R$ gives the atoms of the decomposition in which $M$ has high density and the $L^2$ error term is small.

As in the counting lemma for graphs, it will turn out that we do not need the hypothesis of having a \emph{copy} of $N$ in $R_{\varepsilon,\zeta}$, i.e. an \emph{injective} linear map sending $N$ inside $R_{\varepsilon,\zeta}$. Rather, we only need there to be a \emph{homomorphism} from $N$ to $R_{\varepsilon,\zeta}$, i.e. any linear map $\iota$ such that $\iota(N)\subseteq R_{\varepsilon,\zeta}$.

\begin{thm}[Counting Lemma]\label{thm:counting}
For every matroid $N$, positive real number $\zeta$, and integer $d\geq |N|-2$, there exist positive real numbers $\beta$ and $\varepsilon_0$, a positive nonincreasing function $\eta:\ZZ^+\to \RR^+$, and positive nondecreasing functions $r,\nu:\ZZ^+\to \ZZ^+$ such that the following holds for all $\varepsilon\leq \varepsilon_0$. Let $M\subseteq \FF_2^n\setminus \{0\}$ be a matroid with an $(\eta,r,d)$-regular partition $f_1+f_2+f_3$ with a corrresponding factor $\cB$. If $n\geq \nu(|\cB|)$ and there exists a homomorphism from $N$ to the reduced matroid $R_{\varepsilon,\zeta}$, then there exist at least $\beta \frac{(2^n)^{r(N)}}{\|\cB\|^{|N|}}$ copies of $N$ in $M$.
\end{thm}

We reiterate that the case where $N$ is an affine matroid (i.e. $\chi(N)=1$) is proved in \cite{VeryCountingMaybe} in the context of property testing; here we prove the lemma in full generality using a very similar argument. The basic idea is to obtain a lower bound for the probability that a linear map $\iota:\FF_2^{r(N)}\to \FF_2^{n}$ chosen uniformly at random sends $N$ to a set contained entirely within $M$, by splitting $f=1_M$ into its three parts and expanding out the product we get in the expectation expression. Letting $N=\{N_1,\dots,N_m\}$ and $r(N)=\ell$ we have

\[\Pr_{\iota:\FF_2^{\ell}\to \FF_2^{n}}[\iota(N)\subseteq M]=\E_{\iota}\left[\prod_{i=1}^m f(\iota(N_i))\right]=\E_{\iota}\left[\sum_{(j_1,\dots,j_m)\in \{1,2,3\}^m}\prod_{i=1}^m f_{j_i}(\iota(N_i))\right].\]

For ease of notation in the proof that follows, we will introduce the concept of a linear form, as used in \cite{VeryCountingMaybe} and \cite{hatamiRegCount}.

\begin{defn}
A \emph{linear form} on $k$ variables is a linear map $L:(\FF_2^n)^k\to \FF_2^n$ of the form $L(x_1,\dots,x_k)=\sum_{i=1}^k \ell_i x_i$, where $\ell_i\in \FF_2\forall i$.
\end{defn}

Note that a linear form on $k$ variables can be thought of as a vector in $\FF_2^k$: if $L$ is given by $L(x_1,\dots,x_k)=\sum_{i=1}^k \ell_i x_i$, we can identify $L$ with $(\ell_1,\dots,\ell_k)\in \FF_2^k$. Conversely, we can think of each element $N_j$ of $N$ as a linear form $L_j$ on $\ell$ variables. Each linear map $\iota: \FF_2^{\ell}\to \FF_2^{n}$ corresponds to a point $X=(x_1,\dots,x_{\ell})\in (\FF_2^{n})^{\ell}$ such that $\iota(N_j)=L_j(X)$. So, instead of taking an expectation over linear maps, we can take the more intuitive approach of taking an expectation over tuples of points. The expression from before is the same as

\begin{align*}
    \Pr_{X\in (\FF_2^{n})^{\ell}}[L_j(X)\in M ~ \forall j \in [1,m]]&=\E_{X}\left[\prod_{i=1}^m f(L_j(X))\right]\\
    &=\E_{X}\left[\sum_{(i_1,\dots,i_m)\in \{1,2,3\}^m}\prod_{i=1}^m f_{i_j}(L_j(X))\right].
\end{align*}

The terms involving the Gowers uniform part $f_2$ are the easiest to deal with; we will simply invoke the following result with $s=d$.

\begin{lemma}[{\cite[Lemma~3.14]{hatamiRegCount}}]
\label{lem:gowerscount}
Let $f_1,\ldots,f_m:\FF_p^n \to \DD$. Let  $N=\{L_1,\dots,L_m\}$ be a system of linear forms in $\ell$ variables. Then for $s\geq m-2$,
$$
\left| \E_{X\in(\FF_p^{n})^{\ell}} \left[ \prod_{j=1}^{m} f_j (L_j(X)) \right]\right| \le \min_{1 \le j \le m} \|f_j\|_{U^{s+1}}.
$$
\end{lemma}

To deal with the remaining terms, we will use a near-orthogonality theorem from \cite{hatamiRegCount}. To state this near-orthogonality theorem, we first need to introduce the notion of \emph{consistency} as defined in \cite{hatamiRegCount}. By a \emph{homogeneous} polynomial over $\FF_p$ we mean a polynomial $P$ such that for all $c\in\FF_p$ there exists a $c'\in\FF_p$ such that $P(cx)\equiv c'P(x)$. In the case of $\FF_2$, this restriction is equivalent to simply requiring $P(0)=0$.

\begin{definition}[Consistency]\label{def:consistent}
Let $N=\{L_1,\dots,L_m\}$ be a system of linear forms in $\ell$ variables. A vector $(\beta_1, \dots, \beta_m) \in \T^m$ is said to be {\em $(d,k)$-consistent with $N$} if there exists a homogeneous polynomial $P$ of degree $d$ and depth $k$ and a point $X\in (\FF_2^n)^\ell$ such that $P(L_j(X))=\beta_j$ for every $j \in [m]$. Let $\Phi_{d,k}(N)$ denote the set of all such vectors.
\end{definition}

Note that $\Phi_{d,k}(N)$ is a subgroup of $\UU_{k+1}^m$: to check that it is closed under addition, observe from Definition~\ref{defn:polynomial} that a linear change of coordinates sends a polynomial $P$ to a polynomial $Q$ of the same degree and depth. Thus, if $\alpha=(\alpha_1,\dots,\alpha_m),\beta=(\beta_1,\dots,\beta_m)\in \Phi_{d,k}(N)$, a change of coordinates yields a single point $X$ such that $P(L_j(X))=\alpha_j, \: Q(L_j(X))=\beta_j \: \forall j$ for some $P,Q$, from which it is clear that $\alpha+\beta\in \Phi_{d,k}(N)$. We define
\begin{align*}
    &\Phi_{d,k}(N)^\perp \\
    & := \left\{(\lambda_1,\ldots,\lambda_m) \in  [0,2^{k+1}-1]^m \ : \ \forall (\beta_1,\ldots,\beta_m) \in  \Phi_{d,k}(N), \ \sum \lambda_j \beta_j = 0   \right\},
\end{align*}

\noindent the set of all $(\lambda_1,\ldots,\lambda_m) \in [0,2^{k+1}-1]^m$ such that $\sum_{k=1}^m \lambda_j P(L_j(X)) \equiv 0$ for every homogeneous polynomial $P$ of degree $d$ and depth $k$ and every point $X$. Call $\Phi_{d,k}(N)^\perp$ the \emph{$(d,k)$-dependency set} of $N$.

\begin{thm}[Near Orthogonality, {\cite[Theorem~3.8]{hatamiRegCount}}]\label{thm:nearOrtho}
Let $N=\{L_1,\dots,L_m\}$ be a system of linear forms in $\ell$ variables, and let $\cB=(P_1,\dots,P_C)$ be an $\varepsilon$-uniform polynomial factor for some $\varepsilon\in (0,1]$ defined only by homogeneous polynomials. For every tuple $\Lambda$ of integers $(\lambda_{i,j})_{i\in [C], j\in [m]}$, define 
$$
P_{\Lambda}(X)= \sum_{i\in [C], j\in [m]} \lambda_{i,j} P_i(L_j(X)).
$$
Then for each $\Lambda$, one of the following two statements holds: 
\begin{itemize}
\item $P_\Lambda(X)= 0$ for all $X$, and furthermore for every $i \in [C]$, we have $(\lambda_{i,j})_{j \in [m]} \in \Phi_{d_i,k_i}(N)^\perp$, where $d_i,k_i$ are the degree and depth of $P_i$, respectively.
\item $P_\Lambda$ is non-constant and $\left| \E_{X\in (\FF_2^n)^\ell} [\expo{P_\Lambda}] \right|< \varepsilon$.
\end{itemize}
\end{thm}

\begin{rem}
Over a general prime-ordered field $\FF_p$, the restriction of homogeneity is dealt with by modifying the decomposition theorem to only use homogeneous polynomials in the factor $\cB$. The situation is even simpler over $\FF_2$, where any polynomial factor can be rewritten in terms of homogeneous polynomials simply by shifting each polynomial by a constant.
\end{rem}

Directly applying Theorem~\ref{thm:nearOrtho} yields the following result, which estimates the number of copies of $N$ with each element in a specified atom of $\cB$.

\begin{theorem}[Near-equidistribution, {\cite[Theorem~3.12]{hatamiRegCount}}]\label{thm:equidist}
Given $\varepsilon > 0$, let $\cB$ be an $\varepsilon$-uniform polynomial factor of
degree $d > 0$ and  complexity $C$ that is defined by a tuple of homogeneous
polynomials $P_1, \dots, P_C: \FF_2^n
\to \T$ having respective degrees $d_1, \dots, d_C$ and
depths $k_1, \dots, k_C$.
Let $N=\{L_1,\dots,L_m\}$ be a system of linear forms on $\ell$ variables.

Suppose $(\beta_{i,j})_{i \in [C], j \in [m]} \in \T^{C \times m}$  is such that $(\beta_{i,1},\ldots,\beta_{i,m}) \in \Phi_{d_i,k_i}(N)$ for every $i \in [C]$. Then
$$
\left|\Pr_{X\in (\FF_2^n)^\ell}\left[P_i(L_j(X)) = \beta_{i,j}~
 \forall i \in [C],j \in [m] \right] - \frac{1}{K}\right| \leq \varepsilon,
$$
where $K= \prod_{i=1}^C |\Phi_{d_i,k_i}(N)|$.
\end{theorem}

In particular, taking $N$ to have a single element yields an estimate on the size of each atom.

\begin{cor}[Size of atoms, {\cite[Lemma~3.2]{VeryCountingMaybe}}]
\label{cor:atomSize}
Given $\varepsilon > 0$, let $\cB$ be an $\varepsilon$-uniform polynomial factor of
degree $d > 0$ and  complexity $C$ that is defined by a tuple of homogeneous
polynomials $P_1, \dots, P_C: \FF_2^n
\to \T$ having respective
depths $k_1, \dots, k_C$.

Suppose $b=(b_1,\dots,b_C) \in \T^{C}$  is such that $b_i \in \UU_{k_i+1}$ for every $i \in [C]$. Then
$$
\left|\Pr_{x}\left[P_i(x) = b_i~
 \forall i \in [C] \right] - \frac{1}{\|\cB\|}\right| \leq \varepsilon.
$$
\end{cor}

The following construction will be useful in allowing us to apply Theorem~\ref{thm:nearOrtho} to the terms we wish to bound. Given a system of linear forms $N=\{L_1,\dots,L_m\}$ on $\ell$ variables, where $L_1(x_1,\dots,x_\ell)=x_1$, let $N'$ be a system of linear forms $\tilde L_1,\dots,\tilde L_{2m-1}$ on $2\ell-1$ variables, where, if $X=(x_1,\dots,x_\ell)$ and $Y=(y_2,\dots,y_\ell)$, then $\tilde L_i(X,Y)=L_i(X)$ for $i\in\{1,\dots,m\}$ and $\tilde L_i(X,Y)=L_{i-m+1}(x_1,Y)$ for $i\in\{m+1,\dots,2m-1\}$. That is, $N'$ corresponds to a rank-$(2\ell-1)$ matroid consisting of the union of two copies of $N$ with $L_1=N_1$ as the only shared element, a so-called \emph{parallel connection} of $N$ with itself at $N_1$. We have the following relation between the sizes of the $(d,k)$-dependency sets of $N$ and $N'$.

\begin{lem}
\label{lem:equivConsistent}
For each $i$, $|\Phi_{d_i,k_i}(N')^\perp|=|\Phi_{d_i,k_i}(N)^\perp|^2$.
\end{lem}
\begin{proof}
The proof is essentially the same as that of Lemma~5.13 in \cite{VeryCountingMaybe}, except that a key step now uses relevant polynomials being homogeneous instead of relevant linear forms being affine.

Consider the map $\varphi: \Phi_{d_i,k_i}(N)^\perp \times \Phi_{d_i,k_i}(N)^\perp \to \Phi_{d_i,k_i}(N')^\perp$ given by
\[\varphi((\lambda_1,\dots,\lambda_m),(\tau_1,\dots,\tau_m))=(\lambda_1 + \tau_1 ,\lambda_2,\dots,\lambda_m, \tau_2,\dots,\tau_m).\]

If $\lambda=(\lambda_1,\dots,\lambda_m),\tau=(\tau_1,\dots,\tau_m)\in \Phi_{d_i,k_i}(N)^\perp$, then $\sum_{j=1}^m \lambda_j P(L_j(X)) = \sum_{j=1}^m \tau_j P(L_j(X)) = 0$ for all $P$ and $X$.  So, for all $(X,Y)$,

\[\sum_{j=1}^{2m-1} \varphi(\lambda,\tau)_j P(\tilde{L}_j(X,Y))=\sum_{j=1}^m \lambda_j P(L_j(X)) + \sum_{j=1}^m \tau_j P(L_j(x_1,Y)) = 0.\]

Thus $\varphi$ indeed maps $\Phi_{d_i,k_i}(N)^\perp \times \Phi_{d_i,k_i}(N)^\perp$ to $\Phi_{d_i,k_i}(N')^\perp$. We claim that $\varphi$ is a bijection, which then implies the desired equality.

Suppose $\lambda\in \Phi_{d_i,k_i}(N)^\perp$. By the definition of a linear form, for each $i$ either $L_i(x_1,0,\dots,0)\equiv 0$ or $L_i(x_1,0,\dots,0)\equiv x_1$. Let $S$ be the set of $i$ such that $L_i(x_1,0,\dots,0)\equiv x_1$. Note that $1\in S$. Setting $x_2=\cdots=x_\ell=0$ and setting $P,x_1$ such that $P$ is a linear polynomial with $P(x_1)=1,P(0)=0$ gives
\[
0=\left(\sum_{j\in S} \lambda_j\right) P(x_1) + \left(\sum_{j\notin S} \lambda_j\right) P(0)=\sum_{j\in S} \lambda_j,
\]
so $\sum_{j\in S} \lambda_j = 0$, meaning $\lambda_1=-\sum_{\substack{j\in S \\ j\neq 1}} \lambda_j$. Thus $\lambda$ is uniquely determined given $\lambda_2,\dots,\lambda_m$, and so both $\lambda$ and $\tau$ are uniquely determined given $\varphi(\lambda,\tau)$, meaning $\varphi$ is injective.

Now suppose $(\lambda_1,\dots,\lambda_m,\tau_2,\dots,\tau_m)\in \Phi_{d_i,k_i}(N')^\perp$, so $\sum_{j=1}^m (\lambda_j Q(L_j(X))+\sum_{j=2}^m \tau_j Q(L_j(x_1,Y)))=0$ for every $Q$ and $(X,Y)$. For convenience, let $\tau_1=\lambda_1$. Similarly to before, setting $x_2=\cdots=x_\ell=0$ gives
\[\sum_{j\in S} \lambda_j Q(x_1)+\sum_{j=2}^m \tau_j Q(L_j(x_1,Y))=0,\]
for any $x_1,Y$ and any homogeneous $Q$, while setting $y_2=\cdots=y_\ell=0$ gives
\[\sum_{j=2}^m \lambda_j Q(L_j(X))+\sum_{j\in S} \tau_j Q(x_1)=0,\]
for any $X$ and any homogeneous $Q$. Here we have used the fact that $Q(0)=0$.

In particular, setting $x_2=\cdots=x_\ell=y_2=\cdots=y_\ell=0$ and setting $Q,x_1$ such that $Q$ is a linear polynomial with $Q(x_1)=1,Q(0)=0$ gives $0 = \sum_{j\in S} \lambda_j + \sum_{j\in S} \tau_j - \lambda_1 = \sum_{\substack{j\in S \\ j\neq 1}} (\lambda_j+\tau_j) + \lambda_1$. So, fixing $X$,

\begin{align*}
& 0=\sum_{j=2}^m \lambda_j Q(L_j(X))+\sum_{j\in S} \tau_j Q(x_1) = \sum_{j=2}^m \lambda_j Q(L_j(X)) - \sum_{j\in S} \lambda_j Q(x_1) + \lambda_1 Q(x_1) \\
& = \left( - \sum_{\substack{j\in S\\ j\neq 1}} \lambda_j\right) Q(x_1)  + \sum_{j=2}^m \lambda_j Q(L_j(X)).
\end{align*}

Thus $\left(- \sum_{\substack{j\in S\\ j\neq 1}} \lambda_j,\lambda_2,\dots,\lambda_m\right)\in \Phi_{d_i,k_i}(N)^\perp$, and $\left(- \sum_{\substack{j\in S\\ j\neq 1}} \tau_j,\tau_2,\dots,\tau_m\right)\in \Phi_{d_i,k_i}(N)^\perp$ likewise. But
\[
    \varphi\left(\left(- \sum_{\substack{j\in S\\ j\neq 1}} \lambda_j,\lambda_2,\dots,\lambda_m\right),\left(- \sum_{\substack{j\in S\\ j\neq 1}} \tau_j,\tau_2,\dots,\tau_m\right)\right)=(\lambda_1,\dots,\lambda_m,\tau_2,\dots,\tau_m).
\]
So $\varphi$ is surjective, and thus bijective as desired.

\end{proof}

Now we proceed with the proof of the Counting Lemma.

\begin{proof}[Proof of Theorem~\ref{thm:counting}]
Let $\ell=r(N)$ and $\alpha(C)=2^{-2dCm}$. We set $r(C)$ to be the integer $r(d,\alpha(C))$ given by Proposition~\ref{prop:regUni} such that every $r$-regular degree $d$ polynomial factor is $\alpha(C)$-uniform. Let $\cB=(P_1,\dots,P_C)$. So, $\cB$ is $\alpha(C)$-uniform. Notice that $\|\cB\|=\prod_{i=1}^C 2^{k_i+1}\leq 2^{dC}$. 

We want a lower bound for the probability that for a linear map $\iota:\FF_2^\ell\to \FF_2^n$ chosen uniformly at random, each element of $N$ is sent inside $M$. Since degenerate maps (ones where the image of $N$ is of lower rank than $N$) are sparse, this will give us that a constant fraction of the copies of $N$ in $\FF_2^n$ (depending on $\|\cB\|$ in an appropriate way) are contained in $M$.

As before, we represent $N$ as a system $\{L_1,\dots,L_m\}$ of linear forms on $\ell$ variables. The probability we want to bound is then

\begin{align*}
    &\Pr_{X\in (\FF_2^{n})^\ell}[L_j(X)\in M ~ \forall j \in [1,m] ]=\E_{X}\left[\prod_{j=1}^m f(L_j(X))\right]\\
    &=\sum_{(i_1,\dots,i_m)\in \{1,2,3\}^m}\E_{X}\left[\prod_{j=1}^m f_{i_j}(L_j(X))\right].
\end{align*}
There are $3^m$ terms in the sum. One of these, the one that will turn out to be the main term, involves only $f_1$. Of the rest, $2^m-1$ involve $f_1$ and $f_3$ but not $f_2$, and the other $3^m-2^m$ terms involve $f_2$.

If one of the $i_j$ is $2$, then since $d\geq m-2$, by Lemma~\ref{lem:gowerscount} we have
\[\left|\E_{X}\left[\prod_{j=1}^m f_{i_j}(L_j(X))\right]\right|\leq \min_{1\leq j\leq m} \|f_{i_j}\|_{U^{d+1}}\leq \|f_2\|_{U^{d+1}}\leq \eta(|\cB|).
\]

We can choose $\eta$ later to make all of these terms sufficiently small. Our probability is thus at least

\[\sum_{(i_1,\dots,i_m)\in \{1,3\}^m}\E_{X}\left[\prod_{j=1}^m f_{i_j}(L_j(X))\right]-3^m \eta(|\cB|).\]

To deal with the remaining terms, we establish a lower bound by only counting within certain ``good'' atoms of the decomposition. Specifically, fix a point $X_0\in \FF_2^n$ such that $L_1(X_0),\dots,L_m(X_0)\in R_{\varepsilon,\zeta}$, i.e. $X_0$ corresponds to a linear map $\iota_{X_0}:\FF_2^\ell\to \FF_2^n$ that acts as a homomorphism from $N$ to the reduced matroid $R_{\varepsilon,\zeta}$. Let $\beta_{i,j}=P_i(L_j(X_0))$, and let $b_j=(\beta_{1,j},\dots,\beta_{C,j})$ be the atom of $\cB$ that $L_j(X_0)$ is in. Define $\mathcal{A}_{X_0}=\{X\in (\FF_2^n)^k:\: L_j(X)=b_j\text{ for all }j\}$. We establish a lower bound by only counting across $X\in \mathcal{A}_{X_0}$. Since $f_1+f_3$ is always nonnegative,

\begin{align*}
    & \sum_{(i_1,\dots,i_m)\in \{1,3\}^m}\E_{X}\left[\prod_{j=1}^m f_{i_j}(L_j(X))\right] \geq\sum_{(i_1,\dots,i_m)\in \{1,3\}^m}\E_{X}\left[1_{\mathcal{A}_{X_0}}\prod_{j=1}^m f_{i_j}(L_j(X))\right] \\
    & =\sum_{(i_1,\dots,i_m)\in \{1,3\}^m}\E_{X}\left[\prod_{j=1}^m f_{i_j}(L_j(X))1_{[\cB(L_j(X))=b_j]}\right].
\end{align*}

We next deal with the main term. Since $f_1\geq 0$, applying Theorem~\ref{thm:equidist} with $\varepsilon=\alpha(C)$ gives
\begin{align*}
    & \E_{X}\left[\prod_{j=1}^m f_{1}(L_j(X)) 1_{[\cB(L_j(X))=b_j]}\right] \\
    &= \Pr_{X\in (\FF_2^n)^\ell}\left[P_i(L_j(X)) = \beta_{i,j}~
 \forall i \in [C],j \in [m] \right]\cdot\\
 &\qquad \qquad \E_{X}\left[\prod_{j=1}^m f_{1}(L_j(X))\middle| \forall j\in[m], \cB(L_j(X))=b_j \right] \\
 &\geq \left(\frac{1}{K} - \alpha(C)\right) \zeta^m.
\end{align*}

Now we deal with the terms involving $f_3$, following the argument used in the corresponding part of the proof of Theorem~5.10 in \cite{VeryCountingMaybe}.

Take such a term $\E_{X}\left[\prod_{j=1}^m f_{i_j}(L_j(X))1_{[\cB(L_j(X))=b_j]}\right]$, and suppose $i_k=3$. Without loss of generality we can take a linear transformation of coordinates and assume $L_k(x_1,\dots,x_\ell)=x_1$. We can also assume without loss of generality that $k=1$. Then, since $|f_1|,|f_3|\leq 1$, we have

\begin{align*}
    &\left|\E_{X}\left[\prod_{j=1}^m f_{i_j}(L_j(X))1_{[\cB(L_j(X))=b_j]}\right]\right| \leq \E_{X}\left[ \left|f_{3}(x_1)\right|\prod_{j=1}^m 1_{[\cB(L_j(X))=b_j]}\right] \\
&=\E_{x_1}\left[\left|f_{3}(x_1)\right|1_{[\cB(x_1)=b_1]}\E_{x_2,\dots,x_\ell}\left[ \prod_{j=1}^m 1_{[\cB(L_j(X))=b_j]}\right]\right].
\end{align*}

By Cauchy-Schwarz,

\begin{align*}
    & \left(\E_{X}\left[ \left|f_{3}(x_1)\right|\prod_{j=1}^m 1_{[\cB(L_j(X))=b_j]}\right]\right)^2 \\
    & \leq \E_{x_1}\left[\left|f_{3}(x_1)\right|^2 1_{[\cB(x_1)=b_1]}\right]\E_{x_1}\left(\E_{x_2,\dots,x_\ell}\left[ \prod_{j=1}^m 1_{[\cB(L_j(X))=b_j]}\right]\right)^2.\label{eqn:csterm} \tag{*}
\end{align*}

By Corollary~\ref{cor:atomSize}, $\Pr_{x_1}[\cB(x_1)=b_1]\leq \frac{1}{\|\cB\|}+\alpha(C)$. Thus by condition (1) in the definition of the reduced matroid,

\begin{align*}
    & \E_{x_1}\left[\left|f_{3}(x_1)\right|^2 1_{[\cB(x_1)=b_1]}\right]=\E_{x_1}\left[\left|f_{3}(x_1)\right|^2 \mid x_1\in b_1\right] \Pr_{x_1}[\cB(x_1)=b_1] \\
    & \leq \varepsilon^2 \left(\frac{1}{\|\cB\|}+\alpha(C)\right) \leq \frac{2\varepsilon^2}{\|\cB\|}.
\end{align*}

Let $Y=(y_2,\dots,y_\ell)\in(\FF_2^n)^{\ell-1}$, so that $(x_1,Y)$ forms another input to the linear forms $L_j$ such that $L_1(x_1,Y)=L_1(X)=x_1$. The second term in the right hand side of (\ref{eqn:csterm}) expands as

\begin{align*}
    & \E_{x_1}\left(\E_{x_2,\dots,x_l}\left[ \prod_{j=1}^m 1_{[\cB(L_j(X))=b_j]}\right]\right)^2
    = \E_{x_1}\left(\E_{x_2,\dots,x_l}\left[ \prod_{\substack{i\in [C]\\ j\in [m]}} 1_{[P_i(L_j(X))=\beta_{i,j}]}\right]\right)^2 \\
    & = \E_{x_1}\left(\E_{x_2,\dots,x_l}\left[ \prod_{\substack{i\in [C]\\ j\in [m]}} \frac{1}{2^{k_i+1}}\sum_{\lambda_{i,j}=0}^{2^{k_i+1}-1}\expo{\lambda_{i,j}(P_i(L_j(X))-\beta_{i,j})}\right]\right)^2 \\
    & = \frac{1}{\|\cB\|^{2m}}\E_{x_1}\left(\sum_{\substack{(\lambda_{i,j})\in\\ \prod_{i,j}[0,2^{k_i+1}-1]}}\expo{-\sum_{\substack{i\in [C]\\ j\in [m]}}\lambda_{i,j}\beta_{i,j}}\E_{x_2,\dots,x_l} \expo{\sum_{\substack{i\in[C]\\ j\in[m]}} \lambda_{i,j}P_i(L_j(X))}\right)^2 .
\end{align*}
Expanding the square, pulling the expectation over $x_1$ inside the resulting sum, and applying the triangle inequality shows that this is at most
\begin{align*}
    & \frac{1}{\|\cB\|^{2m}}\sum_{\substack{(\lambda_{i,j}),(\tau_{i,j})\in\\ \prod_{i,j}[0,2^{k_i+1}-1]}}\left |\E_{\substack{X,Y}}\left[\expo{\sum_{\substack{i\in[C]\\ j\in[m]}} \lambda_{i,j}P_i(L_j(X))}\expo{\sum_{\substack{i\in[C]\\ j\in[m]}} \tau_{i,j}P_i(L_j(x_1,Y))}\right]\right|.
\end{align*}

To bound this expression, we will interpret it as an expectation of the form for which Theorem~\ref{thm:nearOrtho} gives upper bounds, in terms of the system of linear forms $N'$ constructed before Lemma~\ref{lem:equivConsistent}. Recall that $N'=\{\tilde L_1,\dots,\tilde L_{2m-1}\}$, where $\tilde L_i(X,Y)=L_i(X)$ for $i\in\{1,\dots,m\}$ and $\tilde L_i(X,Y)=L_{i-m+1}(x_1,Y)$ for $i\in\{m+1,\dots,2m-1\}$. As in Lemma~\ref{lem:equivConsistent}, for $\lambda_i, \tau_i\in [0,2^{k+1}-1]^m$ we define $\varphi(\lambda_i,\tau_i)=(\lambda_{i,1}+\tau_{i,1},\lambda_2,\dots,\lambda_m,\tau_2,\dots,\tau_m)$. Letting $\mu_i = \varphi(\lambda_i,\tau_i)$, we have
\begin{align*}
    & \sum_{\substack{(\lambda_{i,j}),(\tau_{i,j})\in\\ \prod_{i,j}[0,2^{k_i+1}-1]}}\left |\E_{\substack{X,Y}}\left[\expo{\sum_{\substack{i\in[C]\\ j\in[m]}} \lambda_{i,j}P_i(L_j(X))}\expo{\sum_{\substack{i\in[C]\\ j\in[m]}} \tau_{i,j}P_i(L_j(x_1,Y))}\right]\right|\\
    & = \left(\prod_{i=1}^C 2^{k_i+1} \right)\sum_{\substack{(\mu_{i,j})\in\\ \prod_{i\in [C], j\in [2m-1]}[0,2^{k_i+1}-1]}}\left |\E_{X,Y}\left[\expo{\sum_{\substack{i\in[C]\\ j\in[2m-1]}} \mu_{i,j}P_i(\tilde{L}_j(X,Y))}\right]\right|\\
    & \leq \left(\prod_{i=1}^C 2^{k_i+1}\right)^{2m} \alpha(C) + \left(\prod_{i=1}^C 2^{k_i+1}\right)\prod_{i=1}^C |\Phi_{d_i,k_i}(N')^\perp| \\
    & = \|\cB\|^{2m} \alpha(C) + \|\cB\|\prod_{i=1}^C |\Phi_{d_i,k_i}(N')^\perp| =  \|\cB\|^{2m} \alpha(C) + \|\cB\|\prod_{i=1}^C |\Phi_{d_i,k_i}(N)^\perp|^2 \\
    & = \|\cB\|^{2m}\alpha(C)+\|\cB\|\prod_{i=1}^C \left(\frac{(2^{k_i+1})^m}{|\Phi_{d_i,k_i}(N)|}\right)^2 = \|\cB\|^{2m}\left(\alpha(C)+\frac{\|\cB\|}{K^2}\right),
\end{align*}

\noindent where $K=\prod_{i=1}^C |\Phi_{d_i,k_i}(N)|$. Here the third line is an application of Theorem~\ref{thm:nearOrtho}, the fourth line follows by Lemma~\ref{lem:equivConsistent}, and the last line follows by observing, as in~\cite{hatamiRegCount}, that $|\Phi_{d_i,k_i}(N)^\perp||\Phi_{d_i,k_i}(N)|=2^{(k_i+1)m}$. Thus,

\begin{align*}
    & \left(\E_{X}\left[ \left|f_{3}(x_1)\right|\prod_{j=1}^m 1_{[\cB(L_j(X))=b_j]}\right]\right)^2 \leq \varepsilon^2 \left(\frac{2}{\|\cB\|}\right) \left(\alpha(C)+\frac{\|\cB\|}{K^2}\right)\\
    & \leq 2\varepsilon^2 \left(\alpha(C)+\frac{1}{K^2}\right),
\end{align*}
so each term in our original sum that involves $f_3$ but not $f_2$ has magnitude at most
\[ \sqrt{2\varepsilon^2 \left(\alpha(C)+\frac{1}{K^2}\right)}= \varepsilon \sqrt{2\left(\alpha(C)+\frac{1}{K^2}\right)}.\]

Finally, we bring everything together. Note that $K= \prod_{i=1}^C |\Phi_{d_i,k_i}(N)|\leq \|\cB\|^m \leq 2^{dCm}$, so $\alpha(C)\leq \frac{1}{K^2}$. Setting $\eta(C) = \left(\frac{\zeta}{3}\right)^m 2^{-dCm-3}$, $\varepsilon_0=\frac{1}{16}\left(\frac{\zeta}{2}\right)^m$, we have

\begin{align*}
    & \E_{X}\left[\prod_{j=1}^m f(L_j(X))\right] \geq \left(\frac{1}{K} - \alpha(C)\right) \zeta^m - 2^{m+\frac{1}{2}} \varepsilon \sqrt{\alpha(C)+\frac{1}{K^2}} - 3^m \eta(|\cB|) \\
    & \geq \frac{1}{2K} \zeta^m - 2^{m+1} \varepsilon \frac{1}{K} - \frac{1}{8K}\zeta^m \\
    & \geq \zeta^m \frac{1}{4K} \geq \frac{\zeta^m}{4} \frac{1}{\|\cB\|^m}.
\end{align*}

So, there are at least $\frac{\zeta^m}{4} \frac{(2^n)^\ell}{\|\cB\|^m}$ linear maps $\iota$ such that $\iota(N)\subseteq M$. At most $\ell (2^n)^{\ell-1}2^{\ell-1}$ such linear maps are not injections. When $n$ is sufficiently large compared to $\ell$ and $C$, this is negligible compared to the number of linear maps we obtained, so for some choice of $\nu$, the number of injections $\iota$ taking $N$ into $M$ is at least $\frac{\zeta^m}{5} \frac{(2^n)^\ell}{\|\cB\|^m}$ for large enough $n$. Since $N$ has at most $2^{\ell^2}$ automorphisms, there are at least $\frac{\zeta^m}{5\cdot 2^{\ell^2}} \frac{(2^n)^\ell}{\|\cB\|^m}$ distinct copies of $N$ in $M$. Taking $\beta=\frac{\zeta^m}{5\cdot 2^{\ell^2}}$, this concludes the proof of the Counting Lemma.

\end{proof}

\section{Basic Applications} \label{sec:apply}
\subsection{The Removal Lemma}
As a first application of our Counting Lemma, we give a simple proof of a removal lemma for matroids. This removal lemma is a special case of a removal lemma for linear equations that appears in \cite{hyperRemoval} and \cite{MatroidRemoval}, in both cases proven using a hypergraph removal lemma.

We say that a matroid $M\subseteq \FF_2^{r(M)}\setminus \{0\}$ is \emph{$\varepsilon$-far from being $N$-free} if for any matroid $M'\subseteq \FF_2^{r(M)}\setminus \{0\}$ that does not contain a copy of $N$, $\E_x[|1_M-1_{M'}|]\geq \varepsilon$.

\begin{thm}[Removal Lemma]
For any $\zeta>0$ and matroid $N$ there is a $\alpha>0$ such that if a matroid $M$ of sufficiently high rank $r(M)$ is $\zeta$-far from being $N$-free, then $M$ contains at least $\alpha\cdot (2^{r(M)})^{r(N)}$ copies of $N$.
\end{thm}

\begin{proof}
Let $\zeta'=\frac{\zeta}{4}$, and let $d=|N|-2$. By the Counting Lemma, there exist $\beta, \eta, \varepsilon, r, \nu$ such that if the reduced matroid $R=R_{\varepsilon,\zeta'}$ given by an $(\eta,r,d)$-regular partition of a matroid $M$ with corresponding factor $\cB$ contains a copy of $N$ and $r(M)\geq \nu(|\cB|)$, then $M$ contains at least $\beta \frac{(2^{r(M)})^{r(N)}}{\|\cB\|^{|N|}}$ copies of $N$. Fix such a choice of $\beta,\eta,\varepsilon,r,\nu$.

Let $M$ be a matroid of rank $n$. Suppose that $M$ is $\zeta$-far from being $N$-free.

By Theorem~\ref{thm:decomp}, we have an $(\varepsilon \zeta'^{1/2},\eta,r,d)$-regular partition of $M$, $1_M=f_1+f_2+f_3$, whose corresponding factor $\cB$ has complexity at most $C$, where $C$ depends on only $\zeta',\varepsilon, \eta,r,d$, i.e. depends on only $\zeta, N$. Let $M'=M\cap R_{\varepsilon,\zeta'}$.

The only elements in $M\setminus M'$ are either
\begin{itemize}
    \item[(i)] In an atom $b$ of $\cB$ such that $E[|f_3(x)|^2\mid x\in b]>\varepsilon^2$, or
    \item[(ii)] In an atom $b$ of $\cB$ such that $E[f(x)\mid x\in b]<\zeta'$.
\end{itemize}

Let $S$ be the subset of $\FF_2^{n}$ contained in atoms $b$ of $\cB$ such that $E[|f_3(x)|^2\mid x\in b]>\varepsilon^2$. Then by condition (iv) of Theorem~\ref{thm:decomp},

\[ \varepsilon^2 \zeta' \geq \|f_3\|_2^2=E_x[|f_3(x)|^2] \geq \frac{|S|}{2^n} \varepsilon^2, \]

so $|S|\leq \zeta' 2^n$. Likewise, let $T$ be the subset of $\FF_2^n$ contained in atoms $b$ of $\cB$ such that $E[f(x)\mid x\in b]<\zeta'$. Then $|T\cap M|< \zeta' |T|\leq \zeta' 2^n$.

So, $E_x{|1_M-1_{M'}|}\leq \frac{|S|+|T|}{2^n} < 2\zeta'=\frac{\zeta}{2}$. Since $M$ is $\zeta$-far from being $N$-free, $M'$ cannot be $N$-free, so $M'$ contains a copy of $N$. If we take $n\geq \nu(C)$, then the Counting Lemma argument above yields that $M$ contains at least $\beta \frac{(2^{n})^{r(N)}}{\|\cB\|^{|N|}}\geq \frac{\beta}{2^{dC|N|}}(2^{n})^{r(N)}$ copies of $N$. Since $\beta$ only depends on $N$ and $\zeta$, we are done by taking $\alpha = \frac{\beta}{2^{dC|N|}}$.

\end{proof}

\subsection{The Doubling Lemma and the Geometric Erd\texorpdfstring{\H{o}}{o}s-Stone Theorem}

We extend the argument in the proof of the removal lemma to prove a simple result we will call the \emph{doubling lemma}. To state it, we define the double of a matroid.

\begin{defn}
Let $N$ be a matroid of rank $\ell$. Define its \emph{double} $2N$ to be the matroid of rank $\ell+1$ consisting of the union of $N$ with $\{x+v|x\in N\}$, where $v$ is a nonzero element not contained in the span of the elements of $N$. So, for example, $2BB(n,c)=BB(n+1,c)$. The matroid $2^k N$ is the result of starting with $N$ and doubling $k$ times.
\end{defn}

The matroid $2^k N$ corresponds to the result of replacing each element of $N$ with an affine cube of dimension $k$. Note that if there is a homomorphism from $N$ to $R$, then for any $k>1$ there is a homomorphism from $2^k N$ to $R$ because we can simply first contract each of the affine hypercubes into a point.

\begin{lem}[Doubling Lemma]
For any $\alpha>0$ and matroid $N$ there exists $\alpha'>0$ such that for sufficiently large $n$, if a matroid $M\subseteq \FF_2^n\setminus 0$ contains at least $\alpha (2^n)^{r(N)}$ copies of $N$, then it contains at least $\alpha'(2^n)^{r(N)+1}$ copies of $2N$.
\end{lem}

\begin{proof}
Let $M$ be a matroid of rank $n$, and suppose that $M$ contains at least $\alpha (2^n)^{r(N)}$ copies of $N$. Any element $v\in M$ can be contained in at most $|N|(2^n)^{r(N)-1}$ copies of $N$, so if $M'$ is a subset of $M$ with $|M'| < \frac{\alpha}{|N|}2^n$, then $M\setminus M'$ contains a copy of $N$. So $M$ is $\frac{\alpha}{|N|}$-far from being $N$-free. Set $\zeta\coloneqq \frac{\alpha}{|N|}$.

Let $\zeta'=\frac{\zeta}{4}$, $d=2|N|-2$. By the Counting Lemma, there exist $\beta, \eta, \varepsilon, r, \nu$ such that if there exists a homomorphism from $2N$ to the reduced matroid $R=R_{\varepsilon,\zeta'}$ given by an $(\eta,r,d)$-regular partition of a matroid $M$ with corresponding factor $\cB$, and $r(M)\geq \nu(|\cB|)$, then $M$ contains at least $\beta \frac{(2^{r(M)})^{r(N)+1}}{\|\cB\|^{|2N|}}$ copies of $2N$. Fix such a choice of $\beta,\eta,\varepsilon,r,\nu$.

By Theorem~\ref{thm:decomp}, we have a $(\varepsilon \zeta'^{1/2},\eta,r,d)$-regular partition of $M$, $1_M=f_1+f_2+f_3$, whose corresponding factor $\cB$ has complexity at most $C$, where $C$ depends on only $\zeta',\varepsilon, \eta,r,d$, i.e. depends on only $\zeta, N$. Let $M'=M\cap R_{\varepsilon,\zeta'}$.

By the same argument as in the proof of the removal lemma, $M'$ contains a copy of $N$. Thus there is a homomorphism from $2N$ to $R_{\varepsilon,\zeta'}$.

So, if we take $n\geq \nu(C)$, by the Counting Lemma argument above, $M$ contains at least $\beta \frac{(2^{n})^{r(N)+1}}{\|\cB\|^{|2N|}}\geq \frac{\beta}{2^{2dC|N|}}(2^{n})^{r(N)+1}$ copies of $N$, so we are done by taking $\alpha' = \frac{\beta}{2^{2dC|N|}}$.

\end{proof}

\begin{rem}

We could also have directly gotten a (nondegenerate) copy of $2N$ in $R_{\varepsilon,\zeta'}$ by using an extension of the Chevalley-Warning Theorem to prime power moduli, such as Theorem~B in \cite{powerCW}.

\end{rem}

One particular special case of this result is of interest.

\begin{cor}\label{cor:doublePG}
For any $\alpha>0$ and positive integers $s,t$ there exists $\alpha'>0$ such that for sufficiently large $n$, if a matroid $M\subseteq \FF_2^n\setminus 0$ contains at least $\alpha (2^n)^{s}$ copies of $PG(s-1,2)$, then it contains at least $\alpha'(2^n)^{s+t}$ copies of $BB(s+t,s)$.
\end{cor}

Corollary~\ref{cor:doublePG} should be compared with the following analogous graph-theoretic lemma, used in the proof of results on the chromatic threshold in \cite{chromThresh}.

\begin{lem}[{\cite[Lemma~7]{doublingKn}}]
For every $r, s$ and $\varepsilon > 0$ there exists $\delta=\delta_{r,s}(\varepsilon) > 0$ such that the following holds for sufficiently large $n$. If the $n$-vertex graph $G$ contains at least $\varepsilon n^r$ copies of $K_r$, then $G$ contains at least  $\delta_{r,s}(\varepsilon)n^{rs}$ copies of the Turán graph $K_r(s)$.
\end{lem}

Along the same lines, our Counting Lemma gives a short proof of the Geometric Erd\H{o}s-Stone theorem, Theorem~\ref{thm:gES}, using the Bose-Burton theorem, Theorem~\ref{thm:BB}, analogous to the proof of the Erd\H{o}s-Stone theorem using Turán's theorem.

\begin{proof}[Proof of Theorem~\ref{thm:gES}]
Let $r(M)=n$, $\chi(N)=c$, and $r(N)=\ell$, and suppose $|M|\geq (1-2^{1-c}+\zeta)2^n$. It suffices to show that if $n$ is sufficiently large, $M$ must contain a copy of $BB(\ell,c)$, so without loss of generality $N=BB(\ell,c)$.

Let $\zeta'=\frac{\zeta}{4}$, and let $d=|N|-2$. By the Counting Lemma, there exist $\beta, \eta, \varepsilon, r, \nu$ such that if there is a homomorphism from $N$ to the reduced matroid $R=R_{\varepsilon,\zeta'}$ given by an $(\eta,r,d)$-regular partition of a matroid $M$ with corresponding factor $\cB$, and $r(M)\geq \nu(|\cB|)$, then $M$ contains at least $\beta \frac{(2^{n})^{l}}{\|\cB\|^{|N|}}>0$ copies of $N$. Fix such a choice of $\beta,\eta,\varepsilon,r,\nu$.

By Theorem~\ref{thm:decomp}, we have a $(\varepsilon \zeta'^{1/2},\eta,r,d)$-regular partition of $M$, $1_M=f_1+f_2+f_3$, whose corresponding factor $\cB$ has complexity at most $C$, where $C$ depends on only $\zeta',\varepsilon, \eta,r,d$, i.e. depends on only $\zeta, N$. Let $M'=M\cap R_{\varepsilon,\zeta'}$. As in the proof of the Removal Lemma, we see that $E_x{|1_M-1_{M'}|} < \frac{\zeta}{2},$ so that $|M'|\geq (1-2^{1-c}+\frac{\zeta}{2})2^n$. By the Bose-Burton theorem, $M'\subseteq R_{\varepsilon,\zeta'}$ contains a copy of $PG(c-1,2)$. So there is a homomorphism from $N=2^{\ell-c}PG(c-1,2)$ to $R_{\varepsilon,\zeta'}$, and thus the Counting Lemma gives us at least one copy of $N$ in $M$, as desired.
\end{proof}

\section{Applications to the Critical Threshold Problem} \label{sec:work}

The strong decomposition theorem and our Counting Lemma allow us to extend the arguments using Green's regularity lemma in \cite{main} and \cite{geelenOdd} to address more general cases of Conjecture~\ref{conj:main}. To illustrate the approach we take, we first extend the argument in \cite{main} to the case of $N_{\ell,2,1}$. This case is simple enough that it suffices to only use Green's regularity lemma, but we phrase it in terms of the $d=1$ case of the strong decomposition theorem to highlight the similarity with our approach to a more general case, in the next section.

Recall that for $\ell\geq c+k-1$, $c>1$, $N_{\ell,c,k}$ is defined to be the rank $\ell$ matroid consisting of the union of $BB(\ell,c-1)$ with $k$ linearly independent vectors contained inside the complement of $BB(\ell,c-1)$ in $\FF_2^\ell$. In particular, $N_{\ell,2,1}$ is the union of an affine geometry and a single other point.

\subsection{Verifying the Conjecture for \texorpdfstring{$N_{\ell,2,1}$}{N\_(l,2,1)}}
\label{sec:vernm1}
\begin{prop}
\label{prop:nm1}
$\theta(N_{\ell,2,1})=\frac{1}{4}.$
\end{prop}
\begin{proof}
Fix $\delta>0$. 

Let $\zeta=\frac{\delta}{2}$, and let $d=1$. By the Counting Lemma, there exist $\beta, \eta, \varepsilon, r, \nu$ such that if the reduced matroid $R=R_{\varepsilon,\zeta}$ given by an $(\eta,r,d)$-regular partition of a matroid $M$ with corresponding factor $\cB$ contains a copy of $PG(1,2)$ and $r(M)\geq \nu(|\cB|)$, then $M$ contains at least $\beta \frac{(2^{r(M)})^{2}}{\|\cB\|^{3}}$ copies of $PG(1,2)$. Fix such a choice of $\beta,\eta,\varepsilon,r,\nu$.

Let $M$ be a matroid of rank $n\geq \nu(|\cB|)$ with $|M|\geq (\frac{1}{4}+\delta)2^n$. By Theorem~\ref{thm:decomp}, we have an $(\varepsilon \zeta^{1/2},\eta,r,1)$-regular partition of $M$, $1_M=f_1+f_2+f_3$, whose corresponding factor $\cB$ has complexity at most $C$, where $C$ depends only on $\zeta,\varepsilon,\eta,r,d$, i.e. depends only on $\delta$. Let $R=R_{\varepsilon,\zeta}$ and let $D=R_{\varepsilon,\frac{1}{2}+\zeta}$.
We have two cases, depending on whether $D$ is nonempty.

\emph{Case 1}: $D$ is nonempty; that is, for some atom $b$ of $\cB$, $|b\cap M|\geq (\frac{1}{2}+\zeta)|b|$ and $\E[|f_3(x)|^2\mid x\in b]\leq \varepsilon^2$. Since $\cB$ is a factor defined by linear polynomials, if $h$ is any element of the atom $b_0$ containing $0$, then shifting by $h$ preserves each atom. If $\chi(M)>C$, then there exists such an element $h$ in $M\cap b_0$. Consider the submatroid $M_h=\{w\in M \mid w+h\in M\}$. Since $|M\cap b|\geq (\frac{1}{2}+\zeta)|b|$, $|M_h\cap b|\geq \zeta |b|\geq \frac{\zeta}{\|\cB\|} 2^n$. By the Density Hales-Jewett Theorem, for sufficiently large $n$, $M_h$ contains a copy $A$ of $AG(m-1,2)$. Then $A\cup (A+h)\cup \{h\}$ contains a copy of $N_{\ell,2,1}$. So, either $\chi(M)\leq C$ or $M$ contains a copy of $N_{\ell,2,1}$, as desired.

\emph{Case 2}: $D$ is empty.

Then for each atom $b$ of $\cB$, either $\E[|f_3(x)|^2\mid x\in b]> \varepsilon^2$ or $|b\cap M| < \frac{1}{2}+\zeta$. Since $\|f_3\|_{L^2}\leq \varepsilon\zeta^{1/2}$, the former is true for less than a fraction $\zeta$ of the atoms $b$. We can use this to give a lower bound on the size of $R$. Indeed, we have
\[(\frac{1}{4}+\delta)2^n\leq |M|< (\zeta 2^n + |D|) \cdot 1 + (|R|-|D|)\cdot (\frac{1}{2}+\zeta) + (2^n-|R|)\cdot \zeta,\]
so $|R|> (\frac{1}{2}+2(\delta-2\zeta)) 2^{n}>\frac{1}{2}2^n$.

By Theorem~\ref{thm:BB}, $R$ contains a copy of $PG(1,2)$. So, by the Counting Lemma, $M$ contains at least $\beta \frac{(2^{n})^{2}}{\|\cB\|^{3}}$ copies of $PG(1,2)$. Then some element $h$ is part of at least $\beta \frac{2^{n}}{\|\cB\|^{3}}$ copies of $PG(1,2)$, so $M_h$, as defined in Case 1, has density at least $\frac{\beta}{\|\cB\|^{3}}$. Applying the Density Hales-Jewett Theorem again gives a copy of $AG(\ell-1,2)$ in $M_h$, and thus a copy of $N_{\ell,2,1}$ in $M$, as desired.
\end{proof}

\subsection{An Approach for \texorpdfstring{$N_{\ell,c,1}$}{N\_(l,c,1)}}

The case $c>2$ is more difficult to address because using the Counting Lemma now requires polynomial factors of higher degrees, with which the construction of $M_h$ from before does not interact in as simple a manner. As will be seen, applying our techniques from before to this case leaves us with the problem of showing that a certain polynomial factor $\cB_h$ corresponding to $M_h$ can be chosen to be sufficiently regular. Ensuring this regularity seems out of the reach of our methods as they stand; nevertheless, we give an outline of our approach to highlight some of the new ideas at hand.

\subsubsection{Some helpful results}

Before we begin, we prove two relevant results that may be of use in broader contexts as well. The first is a simple lemma showing that the Gowers norm of a function cannot increase when restricted according to the output of a nonclassical polynomial (of an appropriate degree).

\begin{lem}
\label{lem:gowerspoly}
Let $s\geq 1$. For any (nonclassical) polynomial $P$ of degree at most $s$, constant $\beta \in \TT$, and function $g:\FF_2^n \rightarrow \CC$, we have
\[\|g1_{P(x)=\beta}\|_{U^{s+1}}\leq \|g\|_{U^{s+1}}.\]
\end{lem}

\begin{proof}
Let $P$ have depth $k$. We have
\[g(x)1_{P(x)=\beta}(x)=\frac{1}{2^{k+1}}\sum_{\lambda=0}^{2^{k+1}-1}g(x)\expo{\lambda(P(x)-\beta)}.\]
Since $s+1\geq 2$, the triangle inequality holds for the Gowers norm, so
\begin{align*}
    \|g1_{P(x)=\beta}\|_{U^{s+1}}&=\left\|\frac{1}{2^{k+1}}\sum_{\lambda=0}^{2^{k+1}-1}g(x)\expo{\lambda(P(x)-\beta)}\right\|_{U^{s+1}} \\
    &\leq \frac{1}{2^{k+1}}\sum_{\lambda=0}^{2^{k+1}-1}\left\|g(x)\expo{\lambda(P(x)-\beta)}\right\|_{U^{s+1}}= \|g\|_{U^{s+1}},
\end{align*}
since $\|g\cdot \expo{P(x)}\|_{U^{s+1}}=\|g\|_{U^{s+1}}$ when $P$ has degree $\leq s$.
\end{proof}

The second result is the following proposition, first stated in the introduction.

\begin{repprop}{prop:extBB}
Let $n,c$ be positive integers, let $k_1,\dots,k_n$ be nonnegative integers, and let $G = \bigoplus_{i=1}^n \frac{1}{2^{k_i+1}}\ZZ/\ZZ$. Let $H$ be a subgroup of $G$. Let $M_1,\dots,$ $M_{2^c-1}$ be subsets of $G$. Then there exist $H_1,\dots,H_c\in G/H$, cosets of $H$, such that for $1\leq i\leq c$,
\begin{equation}
\label{eq:extBB}
    \frac{1}{|H|}\sum_{x\in\{0,1\}^{i-1}} \left | M_{2^{i-1}+\sum_{j=1}^{i-1} x_j 2^{j-1}}\cap \left( H_i + \sum_{j=1}^{i-1} x_j H_j \right ) \right | \geq \sum_{j=2^{i-1}}^{2^i-1}\frac{|M_j|}{|G|}.\tag{$*$}
\end{equation}
\end{repprop}

\begin{proof}
We will choose $H_1,\dots,H_c$ in order, greedily. For $1\leq i_0\leq c$, suppose $H_1,\dots,H_{i_0-1}$ have already been chosen such that \eqref{eq:extBB} holds for $1\leq i \leq i_0-1$.
Consider a uniformly random choice of $H_{i_0}\in G/H$. Taking an expectation gives

\begin{align*}
    & \E_{H_{i_0}\in G/H}\frac{1}{|H|}\sum_{x=(x_1,\dots,x_{i_0-1})\in\{0,1\}^{i_0-1}} \left | M_{2^{i_0-1}+\sum_{j=1}^{i_0-1} x_j 2^{j-1}}\cap \left( H_{i_0} + \sum_{j=1}^{i_0-1} x_j H_j \right ) \right | \\
    &= \frac{1}{|H|}\sum_{x=(x_1,\dots,x_{i_0-1})\in\{0,1\}^{i_0-1}} \E_{H'\in G/H}\left | M_{2^{i_0-1}+\sum_{j=1}^{i_0-1} x_j 2^{j-1}}\cap H' \right | \\
    &= \sum_{j=2^{i_0-1}}^{2^{i_0}-1}\frac{|M_j|}{|G|},
\end{align*}
so for some choice of $H_{i_0}$, the inequality \eqref{eq:extBB} holds. Continuing in this fashion, we can successfully pick $H_i$ for all $i\in [1,c]$, as desired.

\end{proof}

As noted in Section~\ref{sec:matrBack}, Theorem~\ref{thm:BB} immediately follows from this result upon setting $G=\FF_2^{r(M)}$, $H$ the trivial subgroup, and $M_1=\cdots=M_{2^c-1}=M$. We can get more control over the copy of $PG(c-1,2)$ thus obtained by applying Proposition~\ref{prop:extBB} iteratively. Given a matroid $M$ with $|M|>2^{r(M)}(1-2^{1-c})$ and a subspace $W$ of $V=\FF_2^{r(M)}$, applying the proposition with $G=V$, $H=W$, $M_1=\cdots=M_{2^c-1}=M$ yields cosets $W_1,\dots,W_{c}$ of $W$ such that, for $1\leq i\leq c$,
\[
\frac{1}{|W|}\sum_{x\in\{0,1\}^{i-1}} \left | M\cap \left( W_i + \sum_{j=1}^{i-1} x_j W_j \right ) \right | \geq 2^{i-1}\frac{|M|}{|V|}.
\]
Now pick representatives $h_1,\dots,h_{c}$ for the cosets $W_1,\dots,W_{c}$. For nonzero $x\in \{0,1\}^{c}$, let $M_{\sum_{j=1}^{c} x_j 2^{j-1}}=(M\cap (\sum_{j=1}^{c} x_j W_j)) - \sum_{j=1}^{c} x_j h_j$. Applying the proposition again with $G=W$, $H$ the trivial subgroup, and the $\{M_j\}_{j=1}^{2^{c}-1}$ just defined, we get points $e_1,\dots,e_{c}\in W$ such that for $1\leq i\leq c$,
\begin{equation}
\label{eq:extBBapp}
    \sum_{x\in\{0,1\}^{i-1}} \left | M_{2^{i-1}+\sum_{j=1}^{i-1} x_j 2^{j-1}}\cap \left \{ e_i + \sum_{j=1}^{i-1} x_j e_j \right \} \right | \geq \sum_{j=2^{i-1}}^{2^i-1}\frac{|M_j|}{|W|}. \tag{$\dagger$}
\end{equation}
Since $|M_{2^i+\sum_{j=1}^{i-1}x_j 2^{j-1}}|=\left|M\cap \left( W_i + \sum_{j=1}^{i-1} x_j W_j \right )\right|$, by the first application of Proposition~\ref{prop:extBB}, the right hand side of~\eqref{eq:extBBapp} is
\[
\sum_{j=2^{i-1}}^{2^i-1}\frac{|M_j|}{|W|}=2^{i-1}\frac{|M|}{|V|}>2^{i-1}-1.
\]
On the other hand, by the definitions of the $M_j$, the left hand side of~\eqref{eq:extBBapp} is equal to
\[
\left |\left\{x\in \{0,1\}^{i-1} \mid (h_i+e_i) + \sum_{j=1}^{i-1} x_j (h_j + e_j) \in \left(M\cap \left(W_i+\sum_{j=1}^{i-1} x_j W_j\right)\right)\right\}\right |.
\]
This quantity must then equal $2^{i-1}$. Thus the elements $\sum_{j=1}^{c} x_j (h_j + e_j)$ for nonzero $x\in \{0,1\}^{c}$ form a copy of $PG(c-1,2)$ in $M$ such that the cosets $H_1,\dots,H_{2^c-1}$ they lie in satisfy
\[
\frac{1}{2^{i-1}}\sum_{j=2^{i-1}}^{2^i-1} \frac{|S\cap H_j|}{|W|}\geq \frac{|M|}{|V|}.
\]

Of relevance to us is the fact that the same argument can be applied when we consider the space of atoms of a polynomial factor instead of the space containing the matroid itself. We will discuss the details of this case where it comes up in the outline below.

\subsubsection{Outline for the case of \texorpdfstring{$N_{\ell,c,1}$}{N\_(l,c,1)}}

We now outline our approach to the $N_{\ell,c,1}$ case of the conjecture. Fixing $c\geq 3$, we wish to show that $\theta(N_{\ell,c,1})=1-3\cdot 2^{-c}$. We begin as in the proof of Proposition~\ref{prop:nm1}. Let $N=N_{\ell,c,1}$, fix $\delta>0$, and let $\zeta=\frac{\delta}{2}$, $d=|N|-2$. By the Counting Lemma, there exist $\beta_1, \eta_1, \varepsilon_1, r_1, \nu_1$ such that if the reduced matroid $R=R_{\varepsilon_1,\zeta}$ given by an $(\eta_1,r_1,d)$-regular partition of a matroid $M$ with corresponding factor $\cB$ contains a copy of $N$ and $r(M)\geq \nu_1(|\cB|)$, then $M$ contains at least $\beta_1 \frac{(2^{r(M)})^{r(N)}}{\|\cB\|^{|N|}}$ copies of $N$. Let $\beta_2,\eta_2,\varepsilon_2,r_2,\nu_2$ be parameters to be chosen later, and define $\beta=\min(\beta_1,\beta_2)$, $\eta(C)=\min(\eta_1(C),\eta_2(C))$, $\varepsilon=\min(\varepsilon_1,\varepsilon_2)$, $r(C)=\max(r_1(C),r_2(C))$, $\nu(C)=\max(\nu_1(C),\nu_2(C))$. Let $\varepsilon'=\frac{1}{2}\varepsilon \zeta^{1/2}$.

Let $M$ be a matroid of rank $n\geq \nu(|\cB|)$ with $|M|\geq (1-3\cdot 2^{-c}+\delta)2^n$. By Theorem~\ref{thm:decomp}, we have a $(\frac{1}{2}\varepsilon'\zeta^{1/2},\eta,r,1)$-regular partition of $M$, $1_M=f_1+f_2+f_3$, whose corresponding factor $\cB$ has complexity at most $C$, where $C$ depends only on $\zeta,\varepsilon,\eta,r,d$, i.e. depends only on $\delta, |N|$. Let $R=R_{\varepsilon',\zeta}$ and let $D=R_{\varepsilon',\frac{1}{2}+\zeta}$. 
As before, we have two cases, depending on the density of $D$.

\emph{Case 1}: $|D|\leq (1-2^{2-c}+\zeta)2^n$.

This case proceeds as in Case~2 in the proof of Proposition~\ref{prop:nm1}. For each atom $b$ of $\cB$, either $\E[|f_3(x)|^2\mid x\in b]> \varepsilon'^2$, $b\subseteq D$, or $|b\cap M| < \frac{1}{2}+\zeta$. By the same argument as in the proof of Proposition~\ref{prop:nm1}, we obtain the lower bound $|R|>(1-2^{1-c}+\zeta)2^n$. By the Bose-Burton Theorem, $R$ contains a copy of $PG(c-1,2)$, and thus there exists a homomorphism from $2^{\ell-c} PG(c-1,2)$ to $R$. Since $N$ is a submatroid of $2^{\ell-c} PG(c-1,2)$, there is a homomorphism from $N$ to $R$. So, since $R\subseteq R_{\varepsilon,\zeta}$, by the Counting Lemma, $M$ contains a copy of $N$, as desired.

\emph{Case 2}: $|D|> (1-2^{2-c}+\zeta)2^n$.

This case turns out to be more difficult than before. Let $\Delta_h P(x)=P(x+h)-P(x)-P(h)$ for any $x,h\in V= \FF_2^n$ and nonclassical polynomial $P$. Assume without loss of generality that $|\cB|=C$. Given an element $h\in V$, let $\cB_h$ be the factor defined by the polynomials $P_1,\dots,P_C,\Delta_h P_1,\dots,\Delta_h P_C$. Note that the indicator function for $D\cap (D+h)$ is constant on each atom of $\cB_h$. In the case of $N_{\ell,2,1}$, we simply had $\cB_h=\cB$.

Let $N^*=BB(\ell-1,c-1)$, so $N=N^*\cup (N^*+v)\cup \{v\}$ for some element $v$. We represent $N^*$ as a system of linear forms on $\ell-1$ variables, $\{L_1,\dots,L_m\}$, where without loss of generality $\{L_1,\dots,L_{2^{c-1}-1}\}$ forms a copy of $PG(c-2,2)$. Let $M_h=\{w\in M \mid w+h\in M\}$. For an appropriately chosen $h$, we seek a lower bound for the number of copies of $N^*$ contained in $M_h$. If $h\in M$, such a lower bound will then yield a copy of $N$ in $M$.

Let $g(x)=f(x)f(x+h)$ be the indicator function for $M_h$, so the expression we wish to give a lower bound for is
\[\E_{X\in (\FF_2^n)^{\ell-1}}\left[ \prod_{j=1}^m g(L_j(X))\right].\]

Note that
\begin{align*}
    f(x)f(x+h)&=(1-f(x))(1-f(x+h))+(f(x)+f(x+h)-1) \\
    &= ((1-f(x))(1-f(x+h))+f_1(x)+f_1(x+h)-1) \\
    & + (f_2(x)+f_2(x+h))+(f_3(x)+f_3(x+h)).
\end{align*}

Let $g_1(x)=(1-f(x))(1-f(x+h))+f_1(x)+f_1(x+h)-1$, $g_2(x)=f_2(x)+f_2(x+h)$, and $g_3(x)=f_3(x)+f_3(x+h)$. So, $g(x)=g_1(x)+g_2(x)+g_3(x)$.

As in the proof of the Counting Lemma, we can obtain a lower bound by only counting within certain ``good'' atoms of $\cB_h$. Specifically, suppose we have a point $X_0\in (\FF_2^n)^{\ell-1}$ such that $L_1(X_0),\dots,L_m(X_0)$ are in atoms $\tilde b_1,\dots,\tilde b_m$ of $\cB_h$ contained in $D\cap (D+h)$ such that $\E[|f_3(x)|^2 \mid x\in \tilde b_j],\E[|f_3(x+h)|^2 \mid x\in \tilde b_j]$ are at most $\varepsilon^2$ for $1\leq j\leq m$. Then

\begin{align*}
    & \E_{X}\left[ \prod_{j=1}^m g(L_j(X))\right]\geq \E_{X}\left[ \prod_{j=1}^m g(L_j(X)) 1_{[\cB_h(L_j(X))=\tilde b_j]}\right]\\
    &= \sum_{(i_1,\dots,i_m)\in \{1,2,3\}^m}\E_X \left[ \prod_{j=1}^m g_{i_j}(L_j(X)) 1_{[\cB_h(L_j(X))=\tilde b_j]} \right].
\end{align*}

The terms where $i_j=2$ for some $j$ can be handled relatively easily. Repeated application of Lemma~\ref{lem:gowerspoly}, followed by the triangle inequality, gives that
\[\|g_2(x) 1_{[\cB_h(x)=\tilde b_j]}\|_{U^{d+1}}\leq \|g_2\|_{U^{d+1}}\leq 2\|f_2\|_{U^{d+1}},\]
\noindent for $1\leq j\leq m$. So, since $\max_x |g_i(x)|\leq 2$ for $1\leq i\leq 3$, applying Lemma~\ref{lem:gowerscount} on $\{\frac{1}{2}g_{i_j}\}_{j=1}^m$ gives

\begin{align*}
    & \left|\E_{X}\left[\prod_{j=1}^m g_{i_j}(L_j(X))1_{[\cB_h(L_j(X))=\tilde b_j]}\right]\right|\leq 2^m \min_{1\leq j\leq m} \left\|\frac{1}{2} g_{i_j}(x)1_{[\cB_h(x)=\tilde b_j]}\right\|_{U^{d+1}} \\
    & \leq 2^m \left\|\frac{1}{2}  g_2(x)1_{[\cB_h(x)=\tilde b_j]}\right\|_{U^{d+1}}\leq 2^{m} \|f_2\|_{U^{d+1}} \leq 2^{m} \eta(|\cB|),
\end{align*}
\noindent for each term where at least one of the $i_j$ is $2$.

Our probability is thus at least

\[\sum_{(i_1,\dots,i_m)\in \{1,3\}^m}\E_{X}\left[\prod_{j=1}^m g_{i_j}(L_j(X))1_{[\cB_h(L_j(X))=\tilde b_j]}\right]-6^m \eta(|\cB|).\]

An important observation is that when $x\in \tilde b_j$, we have $f_1(x),f_1(x+h)\geq \frac{1}{2}+\zeta$, so $g_1(x)=(1-f(x))(1-f(x+h))+f_1(x)+f_1(x+h) - 1 \geq 2\zeta$. Supposing that we can choose $h\in M$ and our factor $\cB$ such that $\cB_h$ is $r'$-regular for $r'$ sufficiently large, we can then finish the argument exactly as in the proof of the Counting Lemma, by choosing $\beta_2,\eta_2,\varepsilon_2,r_2',\nu_2$ depending on $N,d$ appropriately.

The issue that arises is in showing that such a choice of $\cB$ and $h$ exists. For this, our methods do not seem to suffice.

We end this outline by briefly describing how to use Proposition~\ref{prop:extBB} to show that, if $\cB$ and $h$ can be appropriately chosen, then we can find a point $X_0\in (\FF_2^n)^{\ell-1}$ such that $L_1(X_0),\dots,L_m(X_0)$ are in ``good'' atoms of $\cB_h$, as required in the arguments above.

Let $G=\bigoplus_{i=1}^C \frac{1}{2^{k_i+1}}\ZZ/\ZZ$, the image of the map $X\mapsto (P_1(X),\cdots, P_C(X))$. Let $D'$ be the union of the atoms $b$ of $\cB$ that are contained in $D$ and satisfy $\E[|f_3(x+h)|^2 \mid x\in b]\leq \varepsilon'^2$. Since $\|f_3\|_{L^2}\leq \frac{1}{2}\varepsilon'\zeta^{1/2}$, we have $|D'|\geq |D|-\frac{\zeta}{4}2^n > (1-2^{2-c}+\frac{3\zeta}{4})2^n$. We can associate the matroid $D'$ to a subset $S$ of $G$, consisting of the points $g=(g_1,\dots,g_C)$ such that the atom where $P_i=g_i$ for all $i$ is contained in $D'$. By making $\cB$ sufficiently regular, we can ensure that the sizes of atoms are close enough to each other that $\frac{|S|}{|G|}>\frac{|D'|}{2^n}\left(1+\frac{\zeta}{16}\right)^{-1}>1-2^{2-c}+\frac{5}{8}\zeta$.

Let the depths of $\Delta_h P_1,\dots,\Delta_h P_C$ be $k_1',\dots,k_C'$ respectively. Assume for the sake of simplicity that our choice of $h$ satisfies $P_i(h)\in \frac{1}{2^{k_i'+1}}\ZZ/\ZZ$ for $1\leq i\leq C$; it is not difficult to show that this is possible if $h$ can be chosen from the intersection of $M$ with a subspace of sufficiently large codimension. Let $W$ be the subgroup of $G$ isomorphic to $\bigoplus_{i=1}^C \frac{1}{2^{k_i'+1}}\ZZ/\ZZ$, where each term is a subgroup of the corresponding term in $G$. To see how this relates to the problem at hand, consider a fixed tuple $b=(b_1,\dots,b_m)\in G^m$, where $b_j=(b_{j,1},\dots,b_{j,C})$. When $\cB_h$ is sufficiently regular, the set of tuples $b'=(b_1',\dots,b_m')\in G^m$ for which there is a point $X\in (\FF_2^n)^{\ell-1}$ satisfying $P_i(L_j(X))=b_{j,i}$,$P_i(L_j(X)+h)=b_{j,i}'$ for all $i,j$ is exactly those for which $b_j,b_j'$ are in the same coset of $W$ for all $j$.

As previously discussed, the argument from the case of $G=\FF_2^{r(M)}$ using a two-step application of Proposition~\ref{prop:extBB} carries over to this case. The conclusion is that for some cosets $W_1,\dots,W_{c-1}$ of $W$ in $G$ with representatives $h_1,\dots,h_{c-1}$, and some $e_1,\dots,e_{c-1}\in W$, we have $\sum_{j=1}^{c-1} x_j(h_j+e_j)\in S$ for each nonzero $x\in \{0,1\}^{c-1}$, and for $1\leq i\leq c-1$,
\[
\frac{1}{|W|}\sum_{x\in [0,1]^{i-1}}\left|S\cap \left(W_i+\sum_{j=1}^{i-1}x_j W_j\right)\right|\geq 2^{i-1}\frac{|S|}{|G|} > 2^{i-1}\left(1-2^{2-c}+\frac{5}{8}\zeta\right).
\]

Let $b_{\sum_{j=1}^{c-1} x_j 2^{j-1}}=\sum_{j=1}^{c-1} x_j (h_j + e_j)$ for each nonzero $x\in \{0,1\}^{c-1}$. We claim that, for $1\leq i\leq C$, the atoms of $G$ corresponding to $b_1,\dots,b_{2^{c-1}-1}$ are $(d_i,k_i)$-consistent with the system of linear forms $\{L_1,\dots,L_{2^{c-1}-1}\}$ corresponding to $PG(c-2,2)$. Indeed, letting $v_j=h_j+e_j=(v_{j,1},\dots,v_{j,C})$, the function $Q_i(X)=\sum_{j=1}^{c-1} v_{j,i}|x_j|$ is a polynomial of depth at most $k_i$ and degree at most $k_i+1\leq d_i$. Thus, there is a copy of $PG(c-2,2)$ in $D'$ whose elements lie in the atoms $b_1,\dots,b_{2^{c-1}-1}$ of $\cB$.

The next step is to find appropriate atoms $\tilde b_1,\dots,\tilde b_{2^{c-1}-1}$ of $\cB_h$ contained in these atoms of $\cB$. If we fix points $g_1,\dots,g_{2^{c-1}-1}\in W$, where $g_j=(g_{j,1},\dots,g_{j,C})$, then for $j\in [1,2^{c-1}-1]$ we can let $\tilde b_j$ be the atom of $\cB_h$ on which $P_i = b_{j,i}$, $\Delta_h P_i = g_{j,i}+b_{j,i}$ for all $i\in [1,C]$. For $j\in [1,2^{c-1}-1]$, let $S_j$ be the subset of $W$ consisting of the points $g_j$ for which $b_j+g_j\in S$ and the atom $\tilde b_j$ so defined satisfies $\E[|f_3(x)|^2 \mid x\in \tilde b_j],\E[|f_3(x+h)|^2 \mid x\in \tilde b_j]\leq \varepsilon^2$. Again, when $\cB_h$ is sufficiently regular to control the atom sizes well, we can ensure that the number of atoms of $\cB_h$ for which one of the two bounds is exceeded is at most $\frac{\zeta}{2}\frac{1+\frac{\zeta}{16}}{1-\frac{\zeta}{16}}|W|\leq \frac{17}{30}\zeta|W|$. That is,
$|S_j|\geq |(S-b_j)\cap W|-\frac{17}{30}\zeta|W|$. Applying Proposition~\ref{prop:extBB} with $W$ as the group, the trivial group as the subgroup, and sets $S_1,\dots,S_{2^{c-1}-1}$ gives points $p_1,\dots,p_{c-1}\in W$ such that such for $1\leq i\leq c-1$,
\[
    \sum_{x\in\{0,1\}^{i-1}} \left | S_{2^{i-1}+\sum_{j=1}^{i-1} x_j 2^{j-1}}\cap \left \{ p_i + \sum_{j=1}^{i-1} x_j p_j \right \} \right | \geq \sum_{j=2^{i-1}}^{2^i-1}\frac{|S_j|}{|W|}.
\]
The right hand side is
\begin{align*}
    & \sum_{x\in\{0,1\}^{i-1}} \frac{\left | S_{2^{i-1}+\sum_{j=1}^{i-1} x_j 2^{j-1}} \right |}{|W|} \\
    & \geq \sum_{x\in\{0,1\}^{i-1}} \frac{|(S-b_{2^{i-1}+\sum_{j=1}^{i-1} x_j 2^{j-1}})\cap W|-\frac{17}{30}\zeta|W|}{|W|} \\
    & = \sum_{x\in\{0,1\}^{i-1}} \frac{\left | S\cap \left( W_i + \sum_{j=1}^{i-1} x_j W_j \right ) \right | - \frac{17}{30}\zeta|W|}{|W|} \\
    & > 2^{i-1} \left(1-2^{2-c}+\frac{5}{8}\zeta - \frac{17}{30}\zeta|W| \right ) > 2^{i-1}(1-2^{2-c}).
\end{align*}

This is greater than $2^{i-1}-1$. So, by the same argument as before, if for all nonzero $x\in \{0,1\}^{c-1}$ we let $g_{\sum_{j=1}^{c-1} x_j 2^{j-1}} = \sum_{j=1}^{c-1} x_j p_j$, then $g_j\in S_j$ for all $j$. Also as before, for $i\in [1,C]$, $(g_{1,i},\dots,g_{2^{c-1}-1,i})$ is $(d_i',k_i')$-consistent with the system of linear forms $\{L_1,\dots,L_{2^{c-1}-1}\}$. So as long as $r_2'$ is sufficiently large, by Theorem~\ref{thm:equidist}, there is a copy of $PG(c-2,2)$ whose elements are contained in the atoms $\tilde b_1,\dots, \tilde b_{2^{c-1}-1}$ of $\cB_h$, respectively, as desired. So indeed, we are only left with the problem of showing that $\cB$ and $h$ can be chosen so that $\cB_h$ is sufficiently regular.

\section{Future Steps}
The Counting Lemma has potential for giving simple proofs to other extremal results on matroids whose graph theory analogues are proven using the Szemerédi regularity lemma and its associated counting lemma, including giving short new proofs for known results (as for Theorem~\ref{thm:gES}). A search through more such results in extremal graph theory which yield matroid analogues may well prove fruitful.

In terms of the critical threshold problem, the methods used in our approach to the $N_{\ell,c,1}$ case seem to offer a new approach to the problem, that of using a strong form of regularity and attempting to analyze constructions like $M\cap (M+h)$ through counting lemma-type arguments. The same types of arguments seem equally suited for some special subcases of the $i=4$ case of Conjecture~\ref{conj:main}, and could perhaps be adapted for more general cases as well, though some method of dealing with the regularity of constructions like $\cB_h$ will still likely be needed to proceed. Finally, Proposition~\ref{prop:extBB}, as a simple and general result that makes the proof of the Bose-Burton theorem completely transparent, may be applicable to computing critical thresholds independently of the more sophisticated machinery we have developed, among other potential uses.

\section*{Acknowledgements}
This research was conducted at the University of Minnesota Duluth REU and was supported by NSF grant 1358659 and NSA grant H98230-16-1-0026. The author thanks Joe Gallian for suggesting the problem and for helpful comments on the manuscript, and the referees for their careful reading of the manuscript and
for their useful feedback.

\bibliographystyle{acm}

\bibliography{main}

\begin{thebibliography}{10}

\bibitem{doublingKn}
{\sc Allen, P.}
\newblock Dense $ {H} $-free graphs are almost $(\chi ({H})-1) $-partite.
\newblock {\em Electron. J. Combin. 17}, 1 (2010), R21.

\bibitem{chromThresh}
{\sc Allen, P., B{\"o}ttcher, J., Griffiths, S., Kohayakawa, Y., and Morris,
  R.}
\newblock The chromatic thresholds of graphs.
\newblock {\em Adv. Math. 235\/} (2013), 261--295.

\bibitem{VeryCountingMaybe}
{\sc Bhattacharyya, A., Fischer, E., Hatami, H., Hatami, P., and Lovett, S.}
\newblock Every locally characterized affine-invariant property is testable.
\newblock In {\em Proceedings of the forty-fifth annual ACM symposium on Theory
  of computing\/} (2013), ACM, pp.~429--436.

\bibitem{boninqin}
{\sc Bonin, J.~E., and Qin, H.}
\newblock Size functions of subgeometry-closed classes of representable
  combinatorial geometries.
\newblock {\em Discrete Math. 224}, 1-3 (2000), 37--60.

\bibitem{CrapoRota}
{\sc Crapo, H.~H., and Rota, G.-C.}
\newblock {\em On the foundations of combinatorial theory: {C}ombinatorial
  geometries}, preliminary~ed.
\newblock The M.I.T. Press, Cambridge, Mass.-London, 1970.

\bibitem{gES}
{\sc Geelen, J., and Nelson, P.}
\newblock An analogue of the {E}rd{\H{o}}s-{S}tone theorem for finite
  geometries.
\newblock {\em Combinatorica 35}, 2 (2015), 209--214.

\bibitem{geelenOdd}
{\sc Geelen, J., and Nelson, P.}
\newblock Odd circuits in dense binary matroids.
\newblock {\em Combinatorica\/} (2015), 1--7.

\bibitem{main}
{\sc Geelen, J., and Nelson, P.}
\newblock The critical number of dense triangle-free binary matroids.
\newblock {\em J. Combin. Theory Ser. B 116\/} (2016), 238--249.

\bibitem{KnFree}
{\sc Goddard, W., and Lyle, J.}
\newblock Dense graphs with small clique number.
\newblock {\em J. Graph Theory 66}, 4 (2011), 319--331.

\bibitem{gReg}
{\sc Green, B.}
\newblock A {S}zemer{\'e}di-type regularity lemma in abelian groups, with
  applications.
\newblock {\em Geom. Funct. Anal. 15}, 2 (2005), 340--376.

\bibitem{hatamiRegCount}
{\sc Hatami, H., Hatami, P., and Lovett, S.}
\newblock General systems of linear forms: equidistribution and true
  complexity.
\newblock {\em Adv. Math. 292\/} (2016), 446--477.

\bibitem{hyperRemoval}
{\sc Kr{\'a}l', D., Serra, O., and Vena, L.}
\newblock A removal lemma for systems of linear equations over finite fields.
\newblock {\em Israel J. Math. 187}, 1 (2012), 193--207.

\bibitem{Oxley}
{\sc Oxley, J.}
\newblock {\em Matroid theory}, second~ed., vol.~21 of {\em Oxford Graduate
  Texts in Mathematics}.
\newblock Oxford University Press, Oxford, 2011.

\bibitem{powerCW}
{\sc Schanuel, S.~H.}
\newblock An extension of {C}hevalley's theorem to congruences modulo prime
  powers.
\newblock {\em J. Number Theory 6}, 4 (1974), 284--290.

\bibitem{MatroidRemoval}
{\sc Shapira, A.}
\newblock A proof of {G}reen's conjecture regarding the removal properties of
  sets of linear equations.
\newblock {\em J. Lond. Math. Soc.\/} (2010), jdp076.

\bibitem{tao2012inverse}
{\sc Tao, T., and Ziegler, T.}
\newblock The inverse conjecture for the {G}owers norm over finite fields in
  low characteristic.
\newblock {\em Ann. Comb. 16}, 1 (2012), 121--188.

\bibitem{tidor}
{\sc Tidor, J.}
\newblock Dense binary $ {P}{G} (t-1, 2) $-free matroids have critical number $
  t-1$ or $ t$.
\newblock {\em arXiv:1508.07278\/} (2015).

\end{thebibliography}

\end{document}